\documentclass[a4paper,12pt]{amsart}
\input xypic

\usepackage{amssymb}

\def\hbar{\bar{h}}

\def\Rarr#1{\buildrel #1\over \longrightarrow}

\def\iso{\buildrel \sim\over\to}

\def\GA{{\mathfrak{A}}}

\def\Gm{{\mathfrak{m}}}

\def\Gsl{{\mathfrak{sl}}}

\def\CB{{\mathcal{B}}}
\def\CC{{\mathcal{C}}}
\def\CD{{\mathcal{D}}}

\def\CF{{\mathcal{F}}}

\def\CL{{\mathcal{L}}}

\def\CO{{\mathcal{O}}}
\def\CP{{\mathcal{P}}}

\def\CS{{\mathcal{S}}}
\def\CT{{\mathcal{T}}}

\def\BA{{\mathbf{A}}}

\def\BF{{\mathbf{F}}}
\def\BG{{\mathbf{G}}}

\def\BL{{\mathbf{L}}}

\def\BZ{{\mathbf{Z}}}

\def\eps{\varepsilon}

\def\ad{\operatorname{ad}\nolimits}

\def\Aut{\operatorname{Aut}\nolimits}

\def\can{{\mathrm{can}}}

\def\coker{\operatorname{coker}\nolimits}

\def\DPic{\operatorname{DPic}\nolimits}

\def\en{{\mathrm{en}}}
\def\End{\operatorname{End}\nolimits}
\def\Endgr{\operatorname{Endgr}\nolimits}

\def\Ext{\operatorname{Ext}\nolimits}

\def\GL{\operatorname{GL}\nolimits}

\def\hd{\operatorname{hd}\nolimits}

\def\Hom{\operatorname{Hom}\nolimits}

\def\Homgr{\operatorname{Homgr}\nolimits}

\def\id{\operatorname{id}\nolimits}
\def\im{\operatorname{im}\nolimits}

\def\Int{\operatorname{Int}\nolimits}

\def\mMod{\operatorname{\!-mod}\nolimits}

\def\mMOD{\operatorname{\!-Mod}\nolimits}

\def\mMODgr{\operatorname{\!-Modgr}\nolimits}

\def\mModgr{\operatorname{\!-modgr}\nolimits}

\def\mmodgr{\operatorname{\!-modgr}\nolimits}

\def\Pic{\operatorname{Pic}\nolimits}

\def\mprojgr{\operatorname{\!-projgr}\nolimits}

\def\Out{\operatorname{Out}\nolimits}

\def\SL{\operatorname{SL}\nolimits}

\def\soc{\operatorname{soc}\nolimits}

\def\Spec{\operatorname{Spec}\nolimits}

\def\mstab{\operatorname{\!-stab}\nolimits}

\def\mstabgr{\operatorname{\!-stabgr}\nolimits}

\def\StPic{\operatorname{StPic}\nolimits}

\def\ie{{\em i.e.}}

\def\tg{{\tilde{g}}}

\usepackage[francais]{babel}
\newtheorem{thm}{Th\'eor\`eme}[section]
\newtheorem{lemme}[thm]{Lemme}
\newtheorem{cor}[thm]{Corollaire}
\newtheorem{prop}[thm]{Proposition}
\newtheorem{defi}[thm]{D\'efinition}

\theoremstyle{definition}
\newtheorem{rem}[thm]{Remarque}

\newtheorem{question}[thm]{Question}

\def\DPic{\operatorname{DPic}\nolimits}

\xyoption{all}
\begin{document}
\title{Automorphismes, graduations et cat\'egories
triangul\'ees}
\author{Rapha\"el Rouquier}
\maketitle
\tableofcontents

\vskip 1cm

\begin{abstract}
We give a moduli interpretation of the outer automorphism group
$\Out$ of a finite dimensional algebra similar to that of the Picard
group of a scheme. We deduce that the 
connected component of $\Out$
is invariant under derived and stable equivalences.
This allows us to transfer gradings between algebras
and gives rise to conjectural homological constructions of
interesting gradings on block of finite groups with abelian defect.
We give applications to the lifting
of stable equivalences to derived equivalences.
We give a counterpart of the invariance result for smooth
projective varieties~: the product $\Pic^0\rtimes\Aut^0$ is
invariant under derived equivalence.
\end{abstract}

\section{Introduction}

Il est classique qu'un automorphisme $\sigma$ d'une alg\`ebre de 
dimension finie $A$ fournit une
version tordue $A_\sigma$ du bimodule r\'egulier et qu'on obtient ainsi
une bijection entre automorphismes ext\'erieurs et classes d'isomorphismes
de bimodules, libres de rang $1$ commes modules \`a gauche et \`a droite.
Dans ce travail, nous donnons une version g\'eom\'etrique de ce 
r\'esultat (th\'eor\`eme \ref{repOut}).
Le groupe des automorphismes ext\'erieurs repr\'esente
le foncteur qui associe \`a une vari\'et\'e $S$
le quotient du groupe
des classes d'isomorphisme de
$(A^\en\otimes\CO_S)$-bimodules qui sont localement
libres de rang $1$ comme $(A\otimes\CO_S)$-modules et comme
$(A^\circ\otimes\CO_S)$-modules par $\Pic(S\times\Spec ZA)$.
Le r\'esultat classique s'en d\'eduit par passage aux points
ferm\'es (\ie, cas o\`u $S$ est un point).
Ceci est \`a rapprocher de la description de la vari\'et\'e de Picard
d'une vari\'et\'e projective lisse $X$ comme repr\'esentant le
foncteur $S\mapsto \Pic(X\times S)/\Pic(S)$.

\smallskip
Nous d\'eduisons de cette description de $\Out$ que la composante
connexe $\Out^0$ de l'identit\'e est invariante par \'equivalence
de Morita (th\'eor\`eme \ref{equivMorita}),
r\'esultat d\^u \`a Brauer \cite{Po}. De mani\`ere
similaire, on d\'eduit l'invariance de $\Out^0$ par \'equivalence
d\'eriv\'ee (th\'eor\`eme \ref{invarderiv}), r\'esultat obtenu
ind\'ependamment, et par des m\'ethodes diff\'erentes,
par B.~Huisgen-Zimmermann et M.~Saorin \cite{HuiSa}.
Nous donnons aussi une version de ce r\'esultat pour les g\'eom\`etres~:
pour une vari\'et\'e projective lisse, le produit $\Pic^0\rtimes \Aut^0$
est invariant par \'equivalence de cat\'egories d\'eriv\'ees (th\'eor\`eme
\ref{PicAut}).

Le cas d'\'equivalences stables entre alg\`ebres auto-injectives est
plus d\'elicat. La rigidit\'e des modules projectifs est bien connue, celle
de facteurs directs projectifs n'est pas nouvelle non plus. Ces propri\'et\'es
sont de nature locale et nous avons besoin d'un crit\`ere qui nous assure
de la pr\'esence globale d'un facteur direct projectif, \`a
partir d'informations ponctuelles. C'est l'objet
de la proposition \ref{factproj}. On d\'eduit alors l'invariance de $\Out^0$
par \'equivalence stable de type de Morita entre alg\`ebres auto-injectives
(th\'eor\`eme \ref{Out0stable}).

\smallskip
Nous en d\'eduisons la possibilit\'e de transporter des graduations par
de telles \'equivalences~: en particulier, on s'attend, de cette mani\`ere,
\`a obtenir des
graduations int\'eressantes pour les blocs \`a d\'efaut ab\'elien des
groupes finis (car l'alg\`ebre de groupe d'un $p$-groupe fini sur
un corps de caract\'eristique $p$ admet une graduation compatible aux
puissances du radical) ou pour les alg\`ebres de Hecke aux racines de l'unit\'e.

\smallskip

Le bloc principal de la cat\'egorie $\CO$
d'une alg\`ebre de Lie semi-simple complexe peut \^etre
munie d'une ``graduation'' extr\^emement int\'eressante --- les
polyn\^omes de Kazhdan-Lusztig s'interpr\`etent alors comme les
multiplicit\'es gradu\'ees des modules simples apparaissant dans les
modules de Verma. Cette graduation provient de l'\'equivalence avec
une cat\'egorie de faisceaux pervers sur la vari\'et\'e des drapeaux,
o\`u la graduation provient des structures de poids (structure de Hodge
mixte ou action de l'endomorphisme de Frobenius). La graduation a aussi \'et\'e
construite alg\'ebriquement par Soergel \cite{Soe} et notre travail
est en ce sens une continuation de celui de Soergel.

Notre travail montre qu'un ph\'enom\`ene similaire
doit se passer pour les blocs \`a d\'efaut ab\'elien de groupes finis,
malgr\'e l'absence de g\'eom\'etrie pour l'interpr\'eter. L'utilisation
de ces graduations pour prouver la conjecture du d\'efaut ab\'elien de
Brou\'e se heurte au probl\`eme de positivit\'e des graduations
obtenues. Le point clef est le rel\`evement d'\'equivalences stables
en \'equivalences d\'eriv\'ees. Nous montrons que de tels rel\`evements
existent pour les alg\`ebres extension triviale d'une alg\`ebre
h\'er\'editaire (th\'eor\`eme \ref{th:relevhered}). Des r\'esultats
similaires dans le cas o\`u le type de repr\'esentation est fini
sont d\^us \`a Asashiba, qui avait utilis\'e
des m\'ethodes de recouvrement pour relever des \'equivalences stables
en \'equivalences d\'eriv\'ees \cite{As1}.

\vskip 0.2cm
Je remercie M.~Brou\'e, D.~Huybrechts et J.-P.~Serre 
pour leurs commentaires et leurs suggestions, et A.~Yekutieli, pour
des discussions qui ont consid\'erablement aid\'e ma
compr\'ehension de ce travail.

Les r\'esultats de cet article ont \'et\'e expos\'es en 2000-2002 (Paris,
Chicago, Constan\c{t}a, Yale, Tokyo, Osaka, Oberwolfach, Londres, Grenoble)
et annonc\'es
plus r\'ecemment dans \cite[\S 3.1.2]{Rou3}. Une version
pr\'eliminaire de cet article a circul\'e en 2000-2001.

\section{Notations}
On prend pour $k$ un corps alg\'ebriquement clos.
Par vari\'et\'e, on entend un sch\'ema s\'epar\'e r\'eduit
de type fini sur $k$.
Par groupe alg\'ebrique, on entend un sch\'ema en groupe affine lisse
de type fini sur $k$. Si $X$ est un
sch\'ema affine et $x$ un point ferm\'e de $X$, on note $\Gm_x$ l'id\'eal
maximal correspondant de $\Gamma(X)$.

On \'ecrira $\otimes$ pour $\otimes_k$.

Soit $A$ une $k$-alg\`ebre. On note $A^\circ$ l'alg\`ebre oppos\'ee et on pose
$A^\en=A\otimes A^\circ$.
On note $J(A)$ le radical de Jacobson de $A$ et on pose
$J^i(A)=J(A)^i$. On note $\soc^i(A)$ l'annulateur de $J^i(A)$ dans $A$.
On note $A\mMod$ la cat\'egorie des $A$-modules \`a gauche de type fini et
$D^b(A)$ sa cat\'egorie d\'eriv\'ee born\'ee, lorsque $A$ est coh\'erent.

Si $A$ est de dimension finie et $M\in A\mMod$, on note $P_M$ une
enveloppe projective de $M$. On note $\Omega_A M$ le noyau d'une
surjection $P_M\to M$. On pose $\Omega^0_A M=M$ et
$\Omega^i_A M=\Omega_A(\Omega_A^{i-1}M)$ pour $i>0$. Si $I_M$ est
une enveloppe injective de $M$ et $M\to I_M$ est une injection,
on note $\Omega^{-1}_AM$ son conoyau. On d\'efinit par induction
$\Omega^{-i}_AM=\Omega^{-1}_A(\Omega^{-i+1}_AM)$ pour $i\ge 1$.

Pour une alg\`ebre gradu\'ee $A$, la
cat\'egorie $A\mModgr$ d\'esignera la cat\'egorie des $A$-modules
gradu\'es (de m\^eme pour la cat\'egorie stable $A\mstabgr$).
Pour $V$ un $A$-module
gradu\'e, on note $W=V\langle i \rangle$ le $A$-module gradu\'e donn\'e
par $W_j=V_{i+j}$. On note
$\Homgr(V,V')=\bigoplus_i \Hom(V,V'\langle i\rangle)$.
Pour $V$ un $k$-espace vectoriel, on note $V^*=\Hom_k(V,k)$ le dual
$k$-lin\'eaire.

\smallskip
Soit $X$ une vari\'et\'e sur $k$.
Pour $x$ point de $X$, on note $k(x)$ le corps r\'esiduel en $x$
(corps des fractions de $\CO_x/\Gm_x$).
Pour $\CF$ un $\CO_X$-module, on note $\CF(x)=\CF_x\otimes_{\CO_x} k(x)$.
On \'ecrira parfois ``$x\in X$'' pour ``$x$ est un point ferm\'e de $X$''.

\section{Groupes d'automorphismes d'alg\`ebres de dimension finie}

\subsection{Structure}
\subsubsection{}
Soit $k$ un corps alg\'ebriquement clos et $A$ une $k$-alg\`ebre de dimension
finie.

Soit $G$ un groupe alg\'ebrique d'automorphismes de $A$. Soit
$\Delta_A:A\to A\otimes \CO_G$ le morphisme associ\'e.
Alors, $\Delta_A$ est un morphisme d'alg\`ebres (en particulier,
$g\cdot (ab)=(g\cdot a)(g\cdot b)$ pour $g\in G$ et $a,b\in A$). En outre,
on a un diagramme
commutatif
$$\xymatrix{
A\ar[rr]^{\Delta_A}\ar[d]_{\Delta_A} & & A\otimes \CO_G
  \ar[d]^{\id_A\otimes \Delta} \\
A\otimes \CO_G\ar[rr]_{\Delta_A\otimes \id_{\CO_G}} & &
 A\otimes \CO_G\otimes \CO_G
}$$
o\`u $\Delta:\CO_G\to \CO_G\otimes \CO_G$ est la comultiplication
(en particulier, $g\cdot (g'\cdot a)=(gg')\cdot a$ pour $g,g'\in G$ et
$a\in A$).

\subsubsection{}
On a une suite exacte de groupes alg\'ebriques~:
$$1\to 1+J(A)\to A^\times \to (A/J(A))^\times\to 1$$
et $1+J(A)$ est le radical unipotent de $A^\times$.
On a 
une filtration $1+J(A)\supset 1+J^2(A)\supset\cdots$ dont les quotients
successifs sont des groupes unipotents commutatifs $\BG_a^r$.

Soit $S$ une sous-alg\`ebre semi-simple maximale de $A$, {\em i.e.},
l'image d'une section du morphisme d'alg\`ebres $A\to A/J(A)$.
On a $A=S\oplus JA$ et $A^\times=(1+J(A))\rtimes S^\times$.

Soit $\Aut(A)$ le groupe d'automorphismes de $A$ (vu comme sch\'ema
en groupe sur $k$). On note $\Int(A)$ son groupe d'automorphismes int\'erieurs,
image de $A^\times$ par le morphisme de conjugaison
$\ad:A^\times\to \Aut(A),\ a\mapsto (b\mapsto aba^{-1})$.
C'est un sous-groupe ferm\'e distingu\'e connexe.

On a une suite exacte
$$1\to (ZA)^\times\to A^\times \to \Int(A)\to 1.$$
Elle fournit une d\'ecomposition

$$\Int(A)=\left((1+J(A))/(1+J(ZA))\right)\rtimes
 \left(S^\times/((ZA)^\times\cap S^\times)\right).$$

\begin{lemme}
\label{scissionlocale}
Soit $H$ un sous-groupe ferm\'e distingu\'e d'un groupe alg\'ebrique $G$
et $f:G\to G/H$ le morphisme quotient.
Si $H$ est extension de groupes additifs $\BG_a$ et de groupes
multiplicatifs $\BG_m$, alors, $f$ est localement scind\'e comme
morphisme de vari\'et\'es.
\end{lemme}

\begin{proof}
(cf \cite[VII \S 1.6]{Se})
Il suffit de d\'emontrer le lemme lorsque $G/H$ est connexe.
Soit $\eta$ le point g\'en\'erique de $G/H$. Son image
inverse dans $G$ est un espace homog\`ene principal sous $H$,
donc est trivial, \ie, poss\`ede un point rationnel sur $k(\eta)$
(car un espace homog\`ene principal sous un groupe $\BG_a$ ou $\BG_m$ est
trivial). Un tel point fournit
une section rationnelle du morphisme $f$.
\end{proof}

Il r\'esulte du lemme \ref{scissionlocale} que

\begin{prop}
\label{scissionInt}
Le morphisme canonique de vari\'et\'es $A^\times\to\Int(A)$ est localement
scind\'e.
\end{prop}

\subsubsection{}
\label{secOutscinde}
Soit $\Out(A)$ le groupe quotient $\Aut(A)/\Int(A)$.
La composante connexe de l'identit\'e $\Aut^0(A)$ de $\Aut(A)$ est
contenue dans le sous-groupe $\Aut^K(A)$
des \'el\'ements qui fixent les classes
d'isomorphisme de modules simples (=qui agissent int\'erieurement
sur $A/JA$). On note $\Out^K(A)$ le sous-groupe de $\Out(A)$
d\'efinit de mani\`ere similaire.

Si $A=A_1\times A_2$, alors $\Out^0(A)=\Out^0(A_1)\times\Out^0(A_2)$.
Si $A$ est simple, alors $\Out(A)=1$.

\smallskip

Soit $F(A)$ le sous-groupe ferm\'e distingu\'e
de $\Aut^0(A)$ form\'e des
\'el\'ements qui agissent trivialement sur $A/J(A)$.

Puisque le morphisme d'alg\`ebres $A\to A/J(A)$ est scind\'e et que
$\Aut^0(A/J(A))=\Int(A/J(A))$, on d\'eduit que l'on a une suite
exacte scind\'ee

$$1\to F(A)\to \Aut^0(A)\to \Int(A/J(A))\to 1$$
et $\Aut^0(A)=F(A)\cdot \Int(A)$, {\em i.e.}, le morphisme canonique
$F(A)\to\Out^0(A)$ est surjectif.

\smallskip

On a
$$F(A)\cap\Int(A)=\left((1+J(A))/(1+J(ZA))\right)\rtimes
 \left((ZS)^\times/((ZA)^\times\cap S^\times)\right).$$

On d\'eduit du lemme \ref{scissionlocale} que les morphismes canoniques
$1+J(A)\to F(A)\cap\Int(A)$ et
$F(A)\to\Out^0(A)$ sont localement scind\'es. 

En particulier, on a

\begin{prop}
\label{scission}
Le morphisme de vari\'et\'es $\Aut(A)\to\Out(A)$ est localement scind\'e.
\end{prop}

\begin{prop}
On a $\Aut^K(A)=C_{\Aut(A)}(S)\cdot\Int(A)$.
\end{prop}

\begin{proof}
Tout d'abord, les \'el\'ements de $C_{\Aut(A)}(S)$ et $\Int(A)$ fixent les
classes d'isomorphisme de $A$-modules simples.

On prend $\phi\in\Aut^K(A)$. Alors, $\phi(S)$ est conjugu\'ee \`a
$S$ \cite[Th\'eor\`eme 7.3(c)]{Th}. Quitte \`a changer $\phi$
par un automorphisme int\'erieur, on peut supposer $\phi(S)=S$.
Par hypoth\`ese, $\phi$ induit alors un automorphisme int\'erieur de
$S$. Donc, quitte \`a multiplier par un \'el\'ement inversible de $S$,
on se ram\`ene au cas o\`u $\phi$ agit trivialement sur $S$.
\end{proof}

\begin{rem}
Soit $G(A)$ le sous-groupe de $F(A)$ des \'el\'ements qui agissent
trivialement sur $JA/J^2A$. Alors, $G(A)$ agit trivialement sur
$J^iA/J^{i+1}A$ pour tout $i$, donc $G(A)$ est unipotent
(notons au passage que $\Aut(A)$ est contenu dans le sous-groupe parabolique de
$\End_{k\mMod}(A)^\times$ fixateur du drapeau
$A\supset JA\supset J^2A\supset\cdots$
et que $G(A)$ est l'intersection du radical unipotent de ce sous-groupe
parabolique avec $\Aut^0(A)$).

Ainsi, le noyau de l'application canonique
$$F(A)\to\End_{(A/J(A)\otimes (A/J(A))^\circ}(JA/J^2A)^\times$$
est unipotent~: la partie r\'eductive de $\Out(A)$ se retrouve dans
l'action de $F(A)$ sur l'espace cotangent.

Notons enfin que 
$\End_{(A/J(A)\otimes (A/J(A))^\circ}(JA/J^2A)^\times\simeq
\prod_{V,W}\GL(\Ext^1_A(V,W))$, o\`u $V,W$ parcourent un ensemble
de repr\'esentants des classes d'isomorphisme de $A$-modules simples
(lemme \ref{descriptionT}). On dispose de repr\'esentations de
$F(A)$ sur les $\Ext^1_A(V,W)$ et la somme de ces repr\'esentations
a un noyau unipotent. En particulier, si ces $\Ext^1$ sont de dimension
$0$ ou $1$, alors $F(A)$ (et donc $\Out^0(A)$) est r\'esoluble.
\end{rem}

\begin{rem}
Le lecteur int\'eress\'e pourra consid\'erer la version sch\'ematique
$\Aut'$ du groupe des automorphismes.
Soit $A=k[x]/x^2$. Si $k$ n'est pas de caract\'eristique $2$, alors
$\Aut'(A)\simeq \BG_m$ est r\'eduit. Par contre, lorsque $k$ est de
caract\'eristique $2$,
$\Aut'(A)$ est un produit semi-direct $U^{[1]}\rtimes \BG_m$
o\`u $U^{[1]}$ est le noyau de l'endomorphisme de Frobenius
$x\to x^2$ sur $\BG_a$. Dans ce cas, $\Aut'(A)$ n'est pas r\'eduit.
\end{rem}

\begin{rem}
On a une suite de sous-groupes de $\Aut(A)$, correspondant aux
fixateurs des $J^i(A)$.
\end{rem}

\begin{rem}
Notons que les endomorphismes de $k$-alg\`ebres (unitaires) de $A$
pr\'eservent $JA$. En effet, soit $S$ une sous-alg\`ebre semi-simple
maximale de $A$. Soit $\sigma$ un endomorphisme de $A$. Alors,
la composition $S\to A\Rarr{\sigma} A \to A/JA$ est un morphisme
d'alg\`ebres unitaires entre deux alg\`ebres semi-simples isomorphes,
donc est un isomorphisme. Par cons\'equent, le noyau du morphisme
surjectif $A\Rarr{\sigma} A \to
A/JA$ est $JA$, donc $\sigma(JA)\subseteq JA$.

Les endomorphismes non unitaires de $A$ ne pr\'eservent pas
n\'ecessairement $JA$. La repr\'esentation
r\'eguli\`ere de $k[x]/(x^2)$ fournit un plongement
de $k[x]/(x^2)$ dans $M_2(k)$. En prolongeant par $0$, on obtient un
endomorphisme non unitaire de $A=k[x]/(x^2)\oplus M_2(k)$, injectif sur $JA$,
d'image ayant une intersection nulle avec $JA$.
\end{rem}

\subsection{Familles alg\'ebriques d'automorphismes et bimodules}
\label{autobimod}
\subsubsection{}
\label{familles}
Soit $S$ une vari\'et\'e. On note
$\delta:S\to S\times S$ le plongement diagonal.
Se donner un morphisme $\pi$ de $S$ dans 
l'espace des endomorphismes de $k$-espace vectoriel de $A$ revient \`a se
donner un morphisme $\rho:A\to A\otimes \CO_S$
de faisceaux de $k$-espaces vectoriels sur $S$.
Demander que le morphisme $\pi$ soit \`a valeur dans l'espace des
endomorphismes d'alg\`ebre revient \`a demander que $\rho$ soit un
morphisme d'alg\`ebres.
Demander que $f$ soit \`a valeurs dans la vari\'et\'e des endomorphismes
inversibles revient \`a demander que le morphisme
$$\rho\otimes 1:A\otimes \CO_S\xrightarrow{\rho\otimes \id_{\CO_S}}
A\otimes \CO_S\otimes\CO_S\xrightarrow{\id_A\otimes\delta^*}A\otimes\CO_S$$
soit un isomorphisme
({\em i.e.}, pour tout point ferm\'e $x$ de $S$, le morphisme
compos\'e $A\Rarr{\rho} A\otimes \CO_S\to A\otimes k(x)=A$ est un
isomorphisme).

\smallskip
Soit $\CD_A(S)$ l'ensemble des morphismes d'alg\`ebres
$\rho:A\to A\otimes\CO_S$ tels que
$\rho\otimes 1:A\otimes\CO_S\to A\otimes\CO_S$ est un isomorphisme.

On munit $\CD_A(S)$ d'une structure de groupe par
$$\rho\ast\rho':A\xrightarrow{\rho'}A\otimes\CO_S
\xrightarrow{\rho\otimes 1}A\otimes\CO_S.$$

Soit $\phi:S'\to S$ un morphisme de vari\'et\'es. On a alors une
application canonique 
$$\CD_A(S)\to\CD_A(S'),\ \rho\mapsto
(A\xrightarrow{\rho}A\otimes\CO_S\xrightarrow{\id_A\otimes\phi^*}
A\otimes\CO_{S'}).$$

On obtient ainsi un foncteur $\CD_A$ de la cat\'egorie des vari\'et\'es vers
la cat\'egorie des groupes~: il est repr\'esent\'e par $\Aut(A)$.

\subsubsection{}
\label{CA}
Soit $\tilde{\CC}_A(S)$ l'ensemble des couples $(M,f)$ o\`u 
$M$ est un $(A^\en\otimes\CO_S)$-module, libre de rang $1$
comme $(A\otimes\CO_S)$-module et comme $(A^\circ\otimes\CO_S)$-module
et $f$ est un isomorphisme de $(A\otimes\CO_S)$-modules de $M$ vers
$A\otimes\CO_S$.

On d\'efinit un produit par
$(M,f)\ast(M',f')=((M\otimes_A M')\otimes_{\CO_S\otimes\CO_S}\CO_S,f'')$
o\`u $f''=h\otimes \id_{\CO_S}$ et $h$ est l'isomorphisme de
$(A\otimes\CO_S\otimes\CO_S)$-modules
$$h:M\otimes_A M'\xrightarrow{\id_M\otimes f'}M\otimes_A A\otimes\CO_S
\xrightarrow{\can}M\otimes\CO_S\xrightarrow{f\otimes \id_{\CO_S}}
A\otimes\CO_S\otimes\CO_S.$$

Soit $\phi:S'\to S$ un morphisme de vari\'et\'es. On alors une
application canonique 
$$\tilde{\CC}_A(S)\to\tilde{\CC}_A(S'),\ \
(M,f)\mapsto (\phi^*M,f')$$
o\`u $f':\phi^*M\xrightarrow{\phi^*(f)}\phi^*(A\otimes\CO_S)
\xrightarrow{\can}A\otimes\CO_{S'}$.

On dit que $(M,f)$ et $(M',f')$ sont isomorphes si
$f^{-1}f':M'\to M$ est compatible \`a l'action de $A^\circ$, \ie, si c'est
un isomorphisme de $(A^\en\otimes\CO_S)$-modules.
On note $\CC_A(S)$ le groupe des classes d'isomorphisme d'\'el\'ements
de $\tilde{\CC}_A(S)$.

On a construit un foncteur $\CC_A$ de la cat\'egorie des vari\'et\'es vers
la cat\'egorie des groupes.

\begin{prop}
\label{pr:fonctAut}
Le foncteur $\CC_A$ est repr\'esent\'e par $\Aut(A)$.
\end{prop}

\begin{proof}
On construit un isomorphisme $\CD_A\iso\CC_A$.

Soit $\rho\in\CD_A(S)$. On d\'efinit $M=A\otimes\CO_S$ comme
$(A\otimes\CO_S)$-module et l'action de $a\in A^\circ$ est donn\'ee
par multiplication \`a droite par $\rho(a)$. On prend pour
$f$ l'identit\'e. Alors, $(M,f)\in\CC_A(S)$ et on a construit
un morphisme de foncteurs $\CD_A\to\CC_A$.

On construit l'inverse comme suit. Soit $(M,f)\in\CC_A(S)$.
On a un isomorphisme canonique (multiplication \`a droite)
$$A\otimes\CO_S\iso \End_{A\otimes\CO_S}(A\otimes\CO_S).$$
Via $f$, il induit un isomorphisme
$$\alpha:A\otimes\CO_S\iso \End_{A\otimes\CO_S}(M).$$
Finalement, on obtient un morphisme d'alg\`ebres
$$\rho:A\to \End_{A\otimes\CO_S}(M)\Rarr{\alpha^{-1}} A\otimes\CO_S$$
o\`u la premi\`ere fl\`eche est donn\'ee par l'action \`a droite de $A$ sur
$M$. Alors, $\rho\in\CD_A(S)$.
\end{proof}

\subsubsection{}
\label{IntOut}

Soit $\eps:k\to A\otimes A^*\to A\otimes\CO_A\to A\otimes\CO_{A^\times}$ o\`u
la premi\`ere fl\`eche est le morphisme canonique, la seconde fl\`eche
est l'inclusion des fonctions lin\'eaires sur $A$ dans l'espace des
fonctions polynomiales et la troisi\`eme fl\`eche provient de
l'inclusion de $A^\times$ dans $A$.
Alors, le morphisme $A\to A\otimes \CO_{A^\times}$ associ\'e par
la correspondance de \S \ref{familles}
au morphisme de conjugaison $A^\times\to\Aut(A)$ est donn\'e par
$a\mapsto \eps(1)a\eps(1)^{-1}$.

Soit $(M,f)\in\CC_A(S)$. Alors, le morphisme associ\'e
$\pi:S\to \Aut(A)$ se factorise par $A^\times$
si et seulement si
$M\simeq A\otimes\CO_S$. Via un tel isomorphisme, $f$ est la
multiplication 
\`a gauche par un \'el\'ement $\zeta\in\Gamma(A\otimes\CO_S)^\times$.
Le morphisme $\rho$ associ\'e est $a\mapsto \zeta a \zeta^{-1}$.

Plus g\'en\'eralement,
le morphisme $S\to\Aut(A)$ associ\'e \`a
$(M,f)\in\CC_A(S)$ est \`a valeur dans $\Int(A)$ si et seulement
si $M$ est localement isomorphe \`a $A\otimes\CO_S$
(cf proposition \ref{scissionInt}).

Le morphisme d\'eduit $S\to\Out(A)$ ne d\'epend que de $M$, et pas de $f$.

\subsection{Familles d'automorphismes ext\'erieurs}
\subsubsection{}
Commen\c{c}ons par rappeler l'interpr\'etation classique
des automorphismes ext\'erieurs en terme de bimodules (cas ponctuel).

Un $A^\en$-module $M$ est inversible s'il existe un $A^\en$-module $N$ tel
que $M\otimes_A N\simeq N\otimes_A M\simeq A$ comme $A^\en$-module.
Soit $\Pic(A)$ le groupe des classes d'isomorphisme
de $A^\en$-modules inversibles.

Soit $f\in\Aut(A)$. Soit $A_f$ le $A^\en$-module
\'egal \`a $A$ comme $A$-module, o\`u l'action de $A^\circ$ se fait par
multiplication \`a droite pr\'ec\'ed\'ee de $f$: c'est la construction 
de la preuve de la proposition \ref{pr:fonctAut} pour
$S=\Spec k$. Alors, $A_f$ est inversible et
on obtient un morphisme de groupes $\Aut(A)\to\Pic(A)$ qui se factorise
en une injection $\Out(A)\to \Pic(A)$.
Elle identifie $\Out(A)$ aux classes de modules $M$ qui sont libres
de rang $1$ comme $A$-module et comme $A^\circ$-module.
Le sous-groupe $\Out^K(A)$ s'identifie au noyau du morphisme
canonique $\Pic(A)\to\Aut(K_0(A))$, o\`u $K_0(A)$ est le groupe de
Grothendieck de la cat\'egorie des $A$-modules de type fini. En particulier,
$\Out^K(A)$ est un sous-groupe d'indice fini de $\Pic(A)$.

\subsubsection{}
\label{PicAX}
Soit $S$ une vari\'et\'e. On munit la cat\'egorie
des 
$(A^\en\otimes \CO_S)$-modules d'une structure mono\"{\i}dale
donn\'ee par
$M\boxtimes N= M\otimes_{A\otimes\CO_S}N$.
On pose $M^\vee=\Hom_{A\otimes\CO_S}(M,A\otimes\CO_S)$.

On note $\Pic_A(S)$ le groupe des classes d'isomorphisme
de $(A^\en\otimes\CO_S)$-modules inversibles.
On note $\Pic^f_A(S)$ le groupe des classes d'isomorphisme de
$(A^\en \CO_S)$-modules $M$ qui sont
localement libres de rang $1$ comme
$(A\otimes\CO_S)$-module et comme $(A^\circ\otimes\CO_S)$-module. C'est
un sous-groupe de $\Pic_A(S)$.

\smallskip
Soit $M\in\Pic_A(S)$ et soit $x$ un point ferm\'e de $S$. Alors,
$M(x)$ est un $A^\en$-module inversible et
$M(x)\otimes_A -$ induit un automorphisme de $K_0(A)$
(provenant d'une permutation de l'ensemble des classes de modules simples).
Si cet automorphisme est trivial, alors
$M(x)$ est libre de rang $1$ comme $A$-module et comme $A^\circ$-module.
En particulier, si les automorphismes de $K_0(A)$ sont
triviaux pour tout $x$, alors $M\in\Pic^f_A(S)$.

\smallskip
On dispose d'un morphisme canonique
$$\Pic_A(S)\to \Pic_{A/JA}(S),\ M\mapsto A/JA\otimes_A M\otimes_A A/JA$$
se restreignant en
$\Pic_A^f(S)\to \Pic_{A/JA}^f(S)$.
On note $\CB_A(S)$ le noyau de ce morphisme.

\smallskip
Soient $A_1,A_2$ deux $k$-alg\`ebres de dimension finie. On dispose
d'un isomorphisme canonique
$$\CB_{A_1}(S)\times\CB_{A_2}(S)\iso \CB_{A_1\times A_2}(S),\ (M_1,M_2)\mapsto
M_1\oplus M_2.$$

\smallskip
Soit $\phi:S'\to S$ un morphisme de vari\'et\'es. On dispose
d'un morphisme canonique
$\phi^*:\Pic_A(S)\to\Pic_A(S')$ se restreignant en
$\CB_A(S)\to \CB_A(S')$.

\subsubsection{}
Soit $M\in\Pic_A^f(S)$.
Il existe un recouvrement ouvert
$\CF$ de $S$ tel que $M_{|U}$ est libre de rang $1$ comme
$(A\otimes\CO_U)$-module et comme $(A^\circ\otimes\CO_U)$-module, pour tout
$U\in\CF$.
Fixons un isomorphisme de $(A\otimes\CO_U)$-modules
$f_U:M_{|U}\iso A\otimes\CO_U$, pour tout $U\in\CF$.

On dispose alors (cf \S \ref{CA}), pour $U\in\CF$, d'un morphisme
$U\to \Aut(A)$, donc par composition, d'un morphisme $U\to\Out(A)$.
Puisque ces morphismes sont ind\'ependants du choix
des $f_U$ (cf \S \ref{IntOut}),
ils se recollent. On obtient ainsi un morphisme $S\to\Out(A)$.
On a ainsi construit un morphisme de groupes
$\Pic_A^f(S)\to \Hom(S,\Out(A))$.

\begin{lemme}
\label{lemmeOut0}
Il existe $M\in\Pic_A^f(\Out(A))$ induisant l'identit\'e
de $\Out(A)$.
\end{lemme}

\begin{proof}
Soit $S=\Out^0(A)$.
D'apr\`es \S \ref{secOutscinde},
il existe un recouvrement ouvert $\CF$
de $S$ et des morphismes
$\psi_U:U\to F(A)$ relevant l'immersion ouverte $U\to \Out^0(A)$
pour $U\in\CF$. Quitte \`a raffiner $\CF$, on peut
supposer que $(\psi_U)_{|U\cap V}\times(\psi_V)_{|U\cap V}^{-1}:
U\cap V\to F(A)\cap\Int(A)$ se factorise par $\alpha_{U,V}:
U\cap V\to (1+J(A))\rtimes T$,
o\`u $T$ est un sous-tore de $(ZS)^\times$ tel que
$(ZS)^\times=T\times \bigl((ZA)^\times\cap S^\times)\bigr)$
(cf \S \ref{secOutscinde}). Soit $\rho_U:A\to A\otimes \CO_U$ le
morphisme de faisceaux sur $U$ associ\'e \`a $\phi_U$ (cf \S \ref{familles}).

Soit $N_U=A\otimes\CO_U$ comme $(A\otimes\CO_U)$-module, muni
de l'action de $A^\circ$ donn\'ee par multiplication pr\'ec\'ed\'ee de $\rho_U$.
On a un isomorphisme de $(A^\en\otimes\CO_{U\cap V})$-modules
$\phi_{U,V}:(N_U)_{|U\cap V}\iso (N_V)_{|U\cap V}$ donn\'e par
multiplication \`a droite par l'\'el\'ement $\beta_{U,V}$
de $A\otimes\CO_{U\cap V}$ d\'efini par $\alpha_{U,V}$.

Soit $$c_{U,V,W}=(\beta_{U,W})_{|U\cap V\cap W}^{-1}
(\beta_{V,W})_{|U\cap V\cap W}(\beta_{U,V})_{|U\cap V\cap W}.$$
On a $c_{U,V,W}\in 1+J(ZA)\otimes\CO_{U\cap V\cap W}$.
Soit $\zeta$ la classe correspondante dans
$\check{H}^2(S,1+J(ZA)\otimes\CO_S) \subset H^2(S,1+J(ZA)\otimes\CO_S)$.
On a $H^2(S,1+(JZA)^i/(JZA)^{i+1}\otimes\CO_S)=0$ puisque
$1+(JZA)^i/(JZA)^{i+1}\otimes\CO_S$ est un faisceau coh\'erent et
$\Out^0(A)$ est affine.
Par cons\'equent, $H^2(S,1+J(ZA)\otimes\CO_S)=0$, donc
$\zeta=0$. On en d\'eduit qu'il existe un choix des
$\alpha_{U,V}$ tel que $c_{U,V,W}=0$ et les $N_U$ se recollent
en un $(A^\en\otimes\CO_S)$-module $N$. Puisque
$A/JA\otimes_A N_U\otimes_A A/JA\simeq A/JA\otimes\CO_U$ et
$\phi_{U,V}$ induit l'identit\'e
$A/JA\otimes_A (N_U)_{|U\cap V}\otimes_A A/JA\iso
A/JA\otimes_A (N_V)_{|U\cap V}\otimes_A A/JA$, on d\'eduit
que $A/JA\otimes_A N\otimes_A A/JA\simeq A/JA\otimes\CO_S$. Par
cons\'equent, $N\in\CB_A(S)$.
On a montr\'e que $N$ induit
l'injection canonique $\Out^0(A)\to \Out(A)$.

\smallskip
Soit $\CS$ un ensemble de repr\'esentants des classes \`a droite 
$\Out(A)/\Out^0(A)$ (vus comme points ferm\'es de $\Out(A)$).
On a $\Out(A)=\coprod_{g\in \CS}\Out^0(A)g$.
Soit $g\in \CS$ et $\tg\in\Aut(A)$ relevant $g$.
On consid\`ere la multiplication \`a droite
$g^{-1}:\Out^0(A)g\to\Out^0(A)$. Alors,
$g^{-1*}(N)\otimes_A A_\tg\in \Pic_A^f(\Out^0(A)g)$ induit
l'injection canonique $\Out^0(A)g\to\Out(A)$. Par cons\'equent,
$M=\bigoplus_{g\in \CS}(g^{-1*}M)\otimes_A A_g\in\Pic_A^f(\Out(A))$
induit l'identit\'e de $\Out(A)$.
\end{proof}

\smallskip
Soit $M\in\Pic_A^f(S)$ tel qu'il existe un recouvrement ouvert
$\CF$ de $U$ avec des isomorphismes
de $(A^\en\otimes\CO_U)$-modules
$\phi_U:A\otimes\CO_U\iso M_{|U}$, pour $U\in\CF$.
Pour $U,V\in \CF$, soit
$$c_{U,V}=(\phi_U)_{|U\cap V}^{-1}\circ (\phi_V)_{|U\cap V}\in
\End_{A^\en \otimes \CO_{U\cap V}}(A\otimes \CO_{U\cap V})=
(ZA\otimes\CO_{U\cap V})^\times.$$
Les $c_{U,V}$ d\'efinissent une classe de cohomologie de \v{C}ech dans
$\check{H}^1(S,(ZA\otimes\CO_S)^\times)=
H^1(S,(ZA\otimes\CO_S)^\times)$.

R\'eciproquement, un $1$-cocycle de \v{C}ech \`a valeurs dans
$(ZA\otimes\CO_S)^\times$ d\'efinit un recollement des
$A\otimes\CO_U$ en un $(A^\en\otimes\CO_S)$-module dans
$\Pic_A^f(S)$ et on obtient un morphisme canonique
$H^1(S,(ZA\otimes\CO_S)^\times)\to \Pic_A^f(S)$, inverse \`a droite du
pr\'ec\'edent.

\begin{prop}
\label{Pic}
On a un diagramme commutatif de suites horizontales et verticales exactes
et dont les suites horizontales et la suite verticale gauche sont 
scind\'ees
$$\xymatrix{
& 1\ar[d] & 1\ar[d]& 1\ar[d] \\
1\ar[r] & H^1(S,1+JZA\otimes\CO_S) \ar[r]\ar[d] & \CB_A(S)\ar[r]\ar[d] &
\Hom(S,\Out^K(A)) \ar[r]\ar[d]\ar@/_1pc/@{.>}[l] & 1 \\
1\ar[r] & H^1(S,(ZA\otimes\CO_S)^\times) \ar[r]\ar[d] &
 \Pic_A^f(S)\ar[r]\ar[d] & \Hom(S,\Out(A)) \ar[r]\ar[d]\ar@/_1pc/@{.>}[l] & 1 \\
1\ar[r] & H^1(S,(ZA/JZA\otimes\CO_S)^\times) \ar[r]\ar[d]
 \ar@/^1pc/@{.>}[u] &
 \Pic_{A/JA}^f(S)\ar[d] \ar[r] & \Hom(S,\Out(A/JA))\ar[d]\ar[r]
 \ar@/_1pc/@{.>}[l]& 1\\
& 1 & 1 & 1
}$$
\end{prop}

\begin{proof}
L'exactitude de
$$1\to  H^1(S,(ZA\otimes\CO_S)^\times) \to \Pic_A^f(S)\to \Hom(S,\Out(A))$$
r\'esulte de l'\'etude ci-dessus et de \S \ref{IntOut}.

\smallskip
Soit $M\in\Pic_A^f(\Out(A))$ fourni par le lemme \ref{lemmeOut0}.
Soit $\pi:S\to\Out(A)$. Alors, $\pi^*M\in\Pic_A^f(S)$ et on d\'efinit
ainsi un morphisme de groupes $\Hom(S,\Out(A))\to \Pic_A^f(S)$, scission
du morphisme canonique. Ce morphisme se restreint en
$\Hom(S,\Out^K(A))\to \CB_A(S)$.

\smallskip
La suite exacte 
$$1\to 1+JZA\otimes\CO_S \to (ZA\otimes\CO_S)^\times\to
(ZA/J(ZA)\otimes\CO_S)^\times\to 1$$
est scind\'ee~:
$$(ZA\otimes\CO_S)^\times=\left(1+JZA\otimes\CO_S\right)\times
\left(((ZA)^\times\cap S_0^\times)\times\CO_S^\times\right)$$
o\`u $S_0$ est une sous-alg\`ebre semi-simple maximale de $A$.
On en d\'eduit que la suite verticale gauche est exacte et scind\'ee.
\end{proof}

\begin{rem}
\label{trivK0}
Le sous-groupe $\CB_A(S)/H^1(S,1+J(ZA)\otimes\CO_S)$ 
de $\Pic_A(S)/H^1(S,(ZA\otimes\CO_S)^\times)$ consiste en
les \'el\'ements qui induisent
un automorphisme trivial de $K_0(A)$ en tout point ferm\'e.
\end{rem}

\begin{rem}
On a $H^1(S,(ZA\otimes\CO_S)^\times)=\Pic(S\times\Spec ZA)$. Si
$\CL$ est un fibr\'e inversible sur $S\times\Spec ZA$, alors
on lui associe $M=\CL\otimes_{ZA}A\in\Pic_A^f(S)$.
\end{rem}

\subsubsection{}

Soit $S$ une vari\'et\'e. Soit $\overline{\Pic}^f_A(S)$ le quotient du groupe
des classes d'isomorphisme de
$(A^\en\otimes\CO_S)$-bimodules qui sont localement
libres de rang $1$ comme $(A\otimes\CO_S)$-modules et comme
$(A^\circ\otimes\CO_S)$-modules par le sous groupe des bimodules
de la forme $\CL\otimes_{ZA}A$, o\`u $\CL$ est un fibr\'e inversible sur
$S\times\Spec ZA$. La proposition \ref{Pic} d\'ecrit ce foncteur:

\begin{thm}
\label{repOut}
Le foncteur $\overline{\Pic}^f_A$ est repr\'esent\'e par $\Out(A)$.
\end{thm}

\smallskip

Supposons maintenant que $S$ est un groupe alg\'ebrique.
Soit $M\in\Pic_A^f(S)$. Alors, $M$ induit un morphisme de groupes
alg\'ebriques $S\to\Out(A)$ si et seulement si
$m^*M\simeq M\otimes_A M$, o\`u $m:S\times S\to S$ est la multiplication.

\section{Invariance du groupe des automorphismes ext\'erieurs}
Soient $A$ et $B$ deux $k$-alg\`ebres de dimension finie.

\subsection{Equivalences de Morita}

Soient $L$ un $(A\otimes B^\circ)$-module et $L'$ un
$(B\otimes A^\circ)$-module tels que
$$L\otimes_B L'\simeq A\text{ comme }B^\en\textrm{-modules}$$
$$L'\otimes_A L\simeq B\text{ comme }B^\en\textrm{-modules}.$$

Soit $S$ une vari\'et\'e.
Les foncteurs $L'\otimes_A-\otimes_A L$ et $L\otimes_B -\otimes_B L'$
induisent des
isomorphismes inverses $\Psi_L:\Pic_A(S)\iso \Pic_B(S)$
et $\Psi_{L'}:\Pic_B(S)\iso \Pic_A(S)$.

\begin{lemme}
\label{lemmeMorita}
Les isomorphismes $\Psi_L$ et $\Psi_{L'}$
se restreignent en des isomorphismes inverses
$$\xymatrix{H^1(S,1+J(ZA)\otimes\CO_S)
\ar@<0.5ex>[r]^-{\sim}  & \ar@<0.5ex>[l]^-{\sim}
H^1(S,1+J(ZB)\otimes\CO_S)}.$$
Ils induisent
en des isomorphismes inverses
$$\xymatrix{\CB_A(S)/H^1(S,1+J(ZA)\otimes\CO_S)
\ar@<0.5ex>[r]^-{\sim}  & \ar@<0.5ex>[l]^-{\sim}
\CB_B(S)/H^1(S,1+J(ZB)\otimes\CO_S)}.$$
\end{lemme}

\begin{proof}
Soit $M\in\Pic_A(S)$ localement isomorphe \`a $A\otimes\CO_S$. Alors,
$\Psi_L(M)$ est localement isomorphe \`a
$B\otimes\CO_S$, donc $\Psi_L$ et $\Psi_{L'}$ se restreignent en
des isomorphismes inverses
$$\xymatrix{
H^1(S,(ZA\otimes\CO_S)^\times)
\ar@<0.5ex>[r]^-{\sim}  & \ar@<0.5ex>[l]^-{\sim}
H^1(S,(ZB\otimes\CO_S)^\times)}.$$

Pour la deuxi\`eme partie du lemme, il suffit de noter que si 
$M(x)$ induit un automorphisme trivial de $K_0(A)$, alors $\Psi_L(M)(x)$
induit un automorphisme trivial de $K_0(B)$ (cf remarque \ref{trivK0}).
\end{proof}

Le r\'esultat suivant (pour les groupes r\'eduits) est du \`a Brauer
\cite[Corollaire 2.2]{Po}.

\begin{thm}
\label{equivMorita}
Une \'equivalence $A\mMod\iso B\mMod$ induit
un isomorphisme de groupes alg\'ebriques $\Out^K(A)\iso
\Out^K(B)$, donc par restriction un isomorphisme 
$\Out^0(A)\iso \Out^0(B)$.
\end{thm}

\begin{proof}
Le lemme \ref{lemmeMorita} et la proposition \ref{Pic}
fournissent l'isomorphisme $\Out^K(A)\iso \Out^K(B)$.
\end{proof}

\begin{rem}
Soit $A=k\times M_2(k)$ et $B=k\times k$. Alors,
$\Aut(A)=\GL_2(k)$, $\Out(A)=1$,
$\Aut(B)=\BZ/2$ et $\Out(B)=\BZ/2$. On voit dans ce cas
que $\Aut^0(A)\not\simeq\Aut^0(B)$ et que
$\Out(A)\not\simeq\Out(B)$.
\end{rem}

Le th\'eor\`eme montre que l'action par conjugaison de $\Pic(A)$ sur
$\Out^0(A)$ est alg\'ebrique. Ceci permet de munir $\Pic(A)$ d'une structure
de groupe alg\'ebrique (la structure de vari\'et\'e est donn\'ee par l'union disjointe
de copies de $\Out^0(A)$ param\'etr\'ees par $\Pic(A)/\Out^0(A)$). Alors,
l'isomorphisme du th\'eor\`eme \ref{equivMorita}
s'\'etend en un isomorphisme de groupes
alg\'ebriques entre $\Pic(A)$ et $\Pic(B)$.

\newpage
\subsection{Equivalences d\'eriv\'ees}
\subsubsection{}
\begin{lemme}
\label{semicoho}
Soit $C$ un complexe parfait de $\CO_S$-modules et $x\in S$ tel que
$H^i(C\otimes_{\CO_S}^\BL k(x))=0$ pour $i\not=0$. Alors,
il existe un voisinage ouvert $\Omega$ de $x$ tel que 
$H^i(C_{|\Omega})=0$ pour $i\not=0$.
\end{lemme}

\begin{proof}
Soit $f:M\to N$ un morphisme entre $\CO_S$-modules libres de type fini.

Si $f\otimes_{\CO_S} k(x)$ est surjectif, alors $f$ est surjectif dans
un voisinage
ouvert de $x$. Par dualit\'e, on en d\'eduit que si $f\otimes_{\CO_S} k(x)$
est injectif, alors $f$ est une injection scind\'ee dans un voisinage
ouvert de $x$.

Cela d\'emontre le lemme lorsque $C$ est un complexe born\'e de
$\CO_S$-modules libres de type fini. On en d\'eduit alors le lemme
lorsque $C$ est parfait, puisque qu'il existe un voisinage ouvert
$U$ de $x$ tel que $C_{|U}$ est quasi-isomorphe \`a un complexe
born\'e de $\CO_U$-modules libres de type fini.
\end{proof}

\subsubsection{}
\label{secequivderiv}
Soient $L\in D^b(A\otimes B^\circ)$ et
$L'\in D^b(B\otimes A^\circ)$ tels que
$$L\otimes_B^\BL L'\simeq A\text{ dans }D^b(A^\en)$$
$$L'\otimes_A^\BL L\simeq B\text{ dans }D^b(B^\en).$$

Soit $S$ une vari\'et\'e. Soit $\DPic_A(S)$ le groupe des classes
d'isomorphisme d'objets inversibles de $D^b(A^\en\otimes\CO_S)$
(l'inversibilit\'e est d\'efinie comme en \S \ref{PicAX}).
Les foncteurs $L'\otimes_A^\BL -\otimes_A^\BL L$ et
$L\otimes_B^\BL -\otimes_B^\BL L'$ induisent des
isomorphismes inverses $\Psi_L:\DPic_A(S)\iso \DPic_B(S)$
et $\Psi_{L'}:\DPic_B(S)\iso \DPic_A(S)$.

Soit $\Pic_A^0(S)$ le sous-groupe de $\Pic_A^f(S)$ des
\'el\'ements induisant un morphisme $S\to\Out^0(A)$.

\begin{lemme}
\label{lemmederive}
Les isomorphismes $\Psi_L$ et $\Psi_{L'}$
induisent des isomorphismes inverses
$$\xymatrix{\DPic_A(S)/H^1(S,(ZA\otimes\CO_S)^\times)
\ar@<0.5ex>[r]^-{\sim}  & \ar@<0.5ex>[l]^-{\sim}
\DPic_B(S)/H^1(S,(ZB\otimes\CO_S)^\times)}.$$
Ils se restreignent en des isomorphismes inverses
$$\xymatrix{\Pic_A^0(S)/H^1(S,(ZA\otimes\CO_S)^\times)
\ar@<0.5ex>[r]^-{\sim}  & \ar@<0.5ex>[l]^-{\sim}
\Pic_B^0(S)/H^1(S,(ZB\otimes\CO_S)^\times)}.$$
\end{lemme}

\begin{proof}
La premi\`ere assertion est toute aussi imm\'ediate que dans le lemme
\ref{lemmeMorita}.

\smallskip
Soit $S=\Out^0(A)$ et $M\in\CB_A(S)$ correspondant \`a l'injection
$\Out^0(A)\to\Out(A)$.
Soit $N=\Psi_L(M)$. On a $N(1)\simeq B$.
Il r\'esulte du lemme \ref{semicoho} qu'il
existe un voisinage ouvert $U$ de $1$ dans $S$ tel que $N_{|U}$
n'a de l'homologie qu'en degr\'e $0$ et quitte \`a r\'etr\'ecir $U$, on peut
supposer que $H^0(N)_{|U}$ est libre de rang $1$ comme
$(B\otimes\CO_U)$-module et comme $(B^\circ\otimes\CO_U)$-module.

On a $N(g)\otimes_B^\BL N(h)\simeq N(gh)$ pour $g,h\in S$. Si
$N(g)$ et $N(h)$ n'ont de l'homologie qu'en degr\'e $0$ et que celle-ci
est libre de rang $1$ comme $B$-module et comme $B^\circ$-module, alors
$N(gh)$ n'a de l'homologie qu'en degr\'e $0$ et celle-ci
est libre de rang $1$ comme $B$-module et comme $B^\circ$-module. 
Par cons\'equent, puisque $U$ engendre le groupe alg\'ebrique connexe
$\Out^0(A)$, cette propri\'et\'e est vraie pour tout $g\in S$:
$N$ n'a de l'homologie qu'en degr\'e $0$ et $H^0(N)$ est
localement libre de rang $1$ comme
$(B\otimes \CO_S)$-module et comme $(B^\circ\otimes \CO_S)$-module.
Par cons\'equent, $\Psi_L(M)\in\Pic_B^f(S)$.

\smallskip
Tout \'el\'ement de $\Pic_A^0(S)/H^1(S,(ZA\otimes\CO_S)^\times)$ est de la forme
$\phi^*M$, o\`u $\phi:S\to\Out^0(A)$ est un morphisme.
On a $\Psi_L\phi^*M\simeq \phi^*\Psi_L(M)\in \Pic_B^0(S)$ d'apr\`es
ce qui pr\'ec\`ede. On d\'eduit donc
que $\Psi_L$ et $\Psi_{L'}$ induisent des isomorphismes inverses
$$\xymatrix{\Pic_A^0(S)/H^1(S,(ZA\otimes\CO_S)^\times)
\ar@<0.5ex>[r]^-{\sim}  & \ar@<0.5ex>[l]^-{\sim}
\Pic_B^0(S)/H^1(S,(ZB\otimes\CO_S)^\times)}.$$
\end{proof}

Le r\'esultat suivant 
a \'et\'e obtenu ind\'ependemment, et par des m\'ethodes diff\'erentes,
par B.~Huisgen-Zimmermann et M.~Saor\`{\i}n \cite{HuiSa}. Ce r\'esultat
avait \'et\'e \'etabli pour des \'equivalences d\'eriv\'ees particuli\`eres
auparavant \cite{GuSa}.

\begin{thm}
\label{invarderiv}
Une \'equivalence $D^b(A)\iso D^b(B)$ induit
un isomorphisme de groupes alg\'ebriques $\Out^0(A)\iso\Out^0(B)$.
\end{thm}

\begin{proof}
La th\'eorie de Rickard \cite{Ri} fournit des complexes $L$ et $L'$
comme plus haut.
Le th\'eor\`eme r\'esulte alors du lemme \ref{lemmederive} et du th\'eor\`eme
\ref{repOut}.
\end{proof}

\subsubsection{}
Soit $\DPic(A)$ le groupe des classes d'isomorphisme d'objets inversibles
de $D^b(A^\en)$ (ce groupe a \'et\'e introduit dans \cite{RouZi,Ye1}).
Comme l'explique Yekutieli \cite{Ye2}, le th\'eor\`eme pr\'ec\'edent montre que
$\Out^0(A)$ est un sous-groupe distingu\'e de $\DPic(A)$ et
que l'action de $\DPic(A)$ par conjugaison sur $\Out^0(A)$ produit des
automorphismes de groupe alg\'ebrique. Ceci fournit une structure
de groupe localement alg\'ebrique (=sch\'ema en groupes s\'epar\'e localement
de type fini sur $k$) sur l'union disjointe de
copies de $\Out^0(A)$ param\'etr\'ees par le quotient $\DPic(A)/\Out^0(A)$.

On a finalement une famille de groupes (localement) alg\'ebriques, de
composante identit\'e $\Out^0(A)$~:
$$\xymatrix{
\Out^0(A) \ar@{^{(}->}[r]\ar@/_2pc/[rrrr]_{\textrm{distingu\'e}} &
 \Out^K(A) \ar@{^{(}->}[r] \ar@/^2pc/[rr]^{\textrm{distingu\'e}} &
 \Out(A) \ar@{^{(}->}[r] & \Pic(A) \ar@{^{(}->}[r] & \DPic(A)
}$$

\subsection{Equivalences stables \`a la Morita}
\label{secstable}
Soit $S$ une vari\'et\'e.

\subsubsection{Rigidit\'e des modules projectifs}

Soit $A$ une $k$-alg\`ebre de dimension finie.

\begin{lemme}
\label{rigiditeproj}
Soit $M$ un $(A\otimes\CO_S)$-module, localement libre de type fini
comme $\CO_S$-module.
Soit $x$ un point
ferm\'e de $S$ tel que $P=M(x)$ est un $A$-module
projectif.
Alors, il existe un voisinage ouvert $\Omega$ de $x$ tel que 
$M_{|\Omega}$ est isomorphe \`a $P\otimes\CO_\Omega$.
\end{lemme}

\begin{proof}
On peut supposer $S=\Spec R$, une vari\'et\'e affine.
Consid\'erons le morphisme canonique
$f:M\to P=M/\Gm_x M\simeq M\otimes_R R/\Gm_x$.
Le morphisme canonique
$P\otimes R\to P\otimes R/\Gm_x=P$ se factorise par $f$ en
$g:P\otimes R\to M$. Comme $fg$ est surjectif, il existe
un voisinage ouvert $\Omega$ de $x$ tel que $g_{|\Omega}$ est surjectif.
Puisque le rang de $M$ sur $R$ est la dimension de
$M\otimes_R R/\Gm_x$, c'est aussi le rang de $P\otimes R$ sur $R$. Par
cons\'equent,
$g_{|\Omega}$ est un isomorphisme.
\end{proof}

\subsubsection{Rigidit\'e des facteurs projectifs}
Soit $M$ un $(A\otimes\CO_S)$-module de type fini.
On note
$\rho_M$ l'application canonique
$$\rho_M:\Hom_{A\otimes \CO_S}(M,A\otimes\CO_S)\to
\Hom_{A\otimes\CO_S}(M,A/J(A)\otimes\CO_S).$$
On pose $\bar{M}=M/\bigcap_{f\in\im\rho_M}\ker f$. C'est un
$(A/JA\otimes\CO_S)$-module.

Si $P$ est un $A$-module projectif
et $M=P\otimes\CO_S$, alors $\rho_M$ est surjectif et
$\bar{M}=\hd(P)\otimes\CO_S$ (pour un $A$-module $V$, on note
$\hd(V)$ le plus grand quotient semi-simple de $V$).

\subsubsection{Cas ponctuel}
Nous rappelons comment trouver un facteur direct projectif maximal
lorsque $S=\Spec k$.

Le quotient $\bar{M}$ de $M$ est le plus grand tel que l'application
canonique $M\to\bar{M}$ est projective~:

\begin{lemme}
\label{factprojcorps}
Supposons $S=\Spec k$ et soit
$M$ un $(A\otimes\CO_S)$-module de type fini.
Alors, $M$ a un facteur direct projectif non nul si et seulement si
$\rho_M$ est non nulle.

Soit $P\to\bar{M}$ une enveloppe projective de $\bar{M}$. Alors,
l'application canonique $M\to\bar{M}$ se factorise en un morphisme
surjectif $M\to P$ dont le noyau n'a pas de facteur direct projectif non nul.
\end{lemme}

\begin{proof}
Soit $f\in\Hom_A(M,A)$ tel que $\im f\not\subseteq JA$. Alors,
il existe un $A$-module simple $S$ et un morphisme de $A$-modules
$A\to S$ tel que le compos\'e $M\to A\to S$ est non nul.
Ce compos\'e se factorise par une enveloppe projective
$P_S\to S$ de $S$ en un morphisme
$M\to P_S$, qui est surjectif (lemme de Nakayama). Par cons\'equent,
$P_S$ est facteur direct de $M$.

R\'eciproquement, si $M$ est projectif, alors $\bar{M}=\hd(M)$ et
$\rho_M$ est non nul. La deuxi\`eme partie du lemme est claire.
\end{proof}

\subsubsection{Rigidit\'e locale}

Pour le reste de \S \ref{secstable}, nous supposerons $A$ auto-injective
(\ie, $A$ est un $A$-module injectif).

Le r\'esultat suivant est classique (cf \cite[Corollaire 16]{DoFl} et
\cite[Th\'eor\`eme 3.16]{Da}).

\begin{lemme}
\label{factprojglobal}
Soit $M$ un $(A\otimes\CO_S)$-module, localement libre de type fini
comme $\CO_S$-module.
Soit $x$ un point ferm\'e de 
$S$ et soit $P$ un $A$-module projectif facteur direct de $M(x)$.
Alors, il existe un voisinage ouvert $\Omega$ de $x$
tel que $P\otimes\CO_\Omega$ est facteur direct de $M_{|\Omega}$.
\end{lemme}

\begin{proof}
On peut supposer $S=\Spec R$, une vari\'et\'e affine.
Puisque $A$ est auto-injective, $P^*$ est un $A$-module \`a droite
projectif.
On a $\Ext^1_{A\otimes R}(M,P\otimes \Gm_x)\simeq
\Ext^1_R(P^*\otimes_A M,\Gm_x)$
(cf par exemple \cite[\S 2.2.2]{Roucha}).
Puisque $P^*\otimes_A M$ est un $R$-module projectif, on a donc
$\Ext^1_{A\otimes R}(M,P\otimes \Gm_x)=0$.
Par cons\'equent, un morphisme surjectif $M\to P=P\otimes R/\Gm_x$ se
rel\`eve en morphisme $M\to P\otimes \CO_S$. Ce morphisme est
alors surjectif (et donc scind\'e) dans un voisinage ouvert $\Omega$ de $x$.
\end{proof}

\begin{rem}
Le r\'esultat n'est pas vrai pour des alg\`ebres non
auto-injectives, m\^eme lorsque $S=\Spec R$ avec $R$ locale.

Prenons pour $A$ l'alg\`ebre des matrices triangulaires
sup\'erieures $2\times 2$ sur $k$, $R=k[t]_{(t)}$ et
$M=\{\left(\begin{array}{c} a \\ bt\end{array}\right)\ |\ a,b\in R\}$.
Alors, $M$ est un $(A\otimes R)$-module ind\'ecomposable non projectif
et $R$-libre,
mais $M\otimes_R k$ est somme
directe des deux $A$-modules simples, l'un d'eux \'etant projectif.

Plus g\'en\'eralement, soit $A$ une $k$-alg\`ebre de dimension finie
qui n'est pas auto-injective
et soit $P$ un $A$-module projectif ind\'ecomposable
non injectif. Soit $f:P\to I$ une
enveloppe injective et $N=\coker f$. Soit $\zeta\in\Ext^1_A(N,P)$ d\'etermin\'e
par l'extension. Soit
$\xi=t\zeta\in\Ext^1_{A\otimes R}(N\otimes R,P\otimes R)$ et $M$
le $(A\otimes R)$-module extension de $N\otimes R$ par
$P\otimes R$ d\'etermin\'e par cette classe. Alors, le facteur direct
projectif $P$ de $M\otimes_R k$ ne se remonte pas en un facteur
direct projectif de $M$.
\end{rem}

\subsubsection{Cas g\'en\'eral}

Sous certaines hypoth\`eses, on peut ``recoller'' les facteurs
directs projectifs.

\begin{prop}
\label{factproj}
Soit $S$ une vari\'et\'e affine.
Soit $M$ un $(A\otimes\CO_S)$-module, localement libre de type fini
comme $\CO_S$-module
et soit $P$ un $A$-module projectif.
Supposons
\begin{itemize}
\item[(i)] pour tout point ferm\'e $x$, le $A$-module $P$ est
facteur direct de $M(x)$
\item[(ii)] pour tout point g\'en\'erique $\eta$ 
d'une composante irr\'eductible de $S$, le $A\otimes\CO_\eta$-module
$M_\eta$ n'a pas de facteur direct projectif contenant
strictement $P\otimes \CO_\eta$.
\end{itemize}

Alors, il existe un $(A\otimes \CO_S)$-module projectif $Q$ facteur
direct de $M$ tel que $Q(x)\simeq P$ pour tout point ferm\'e $x$.
\end{prop}

\begin{proof}
Soit $x$ un point ferm\'e de $S$. D'apr\`es le lemme \ref{factprojglobal},
il existe un $(A\otimes\CO_x)$-module $M'$ tel que
$M_x\simeq M'\oplus P\otimes\CO_x$. 
Si $\eta$ est un point g\'en\'erique dont l'adh\'erence contient $x$,
alors le morphisme $\rho_{M'_\eta}$ est nul
puisque $M'_\eta$ n'a pas de facteur direct projectif
(lemme \ref{factprojcorps}).
On en d\'eduit que $\rho_{M'}$ est nul, puisque $\CO_x$ est r\'eduit.
Par cons\'equent, $(\bar{M})_x\simeq \overline{(M_x)}\simeq \hd(P)\otimes\CO_x$.

Le $(A\otimes\CO_S)$-module $\bar{M}$ est donc un $\CO_S$-module
projectif. Puisqu'il est $A$-semi-simple, il est somme directe
de modules $S\otimes F_S$, o\`u $S$ d\'ecrit les classes
d'isomorphisme de $A$-modules simples et $F_S$ est un $\CO_S$-module
projectif. Fixons alors un morphisme surjectif
$h:Q=\bigoplus_S P_S\otimes F_S\to \bar{M}$.
Ce morphisme se factorise par la surjection canonique $M\to\bar{M}$
en $g:Q\to M$. Le morphisme $\Hom_{\CO_S}(g,\CO_S)$ est surjectif,
car il l'est en tout point ferm\'e. Puisque $\Hom_{\CO_S}(Q,\CO_S)$ est
un $(A\otimes\CO_S)$-module projectif,
le morphisme $\Hom_{\CO_S}(g,\CO_S)$ est donc une surjection
scind\'ee et finalement $g$ est une injection scind\'ee.
\end{proof}

\begin{rem}
L'hypoth\`ese aux points g\'en\'eriques n'est pas superflue.

Soit $S=U\cup V$ un recouvrement ouvert de $S$ avec $U,V\not=S$ et $M',M''$
des $(A\otimes\CO_S)$-modules ind\'ecomposables, $\CO_S$-projectifs mais non
$(A\otimes\CO_S)$-projectifs, tels que
$M'_{|U}\simeq P\otimes\CO_U$ et $M''_{|V}\simeq P\otimes\CO_V$
pour un $A$-module projectif $P$.
Alors, $M=M'\oplus M''$ ne poss\`ede pas de facteur direct projectif,
bien que $P$ soit facteur direct de $M(x)$ pour tout point ferm\'e
$x$.
\end{rem}

Le lemme suivant montre qu'il n'est pas n\'ecessaire
de travailler avec un sch\'ema r\'eduit.

\begin{lemme}
\label{nilp}
Soit $R$ une $k$-alg\`ebre commutative, $I$ un id\'eal nilpotent de
$R$ et $M$ un $(A\otimes R)$-module de type fini, projectif comme $R$-module.
Alors, tout facteur direct projectif de $M\otimes_R R/I$ se rel\`eve 
en un facteur direct projectif de $M$.
\end{lemme}

\begin{proof}
Un $(A\otimes R/I)$-module projectif ind\'ecomposable de type fini est
de la forme $P\otimes L$ o\`u $P$ est un $A$-module projectif ind\'ecomposable
et $L$ un $R/I$-module projectif de type fini. Puisque tout
idempotent se rel\`eve \`a travers le morphisme canonique
$\End_R(R^n)\to \End_{R/I}((R/I)^n)$, alors, il existe un
$R$-module projectif de type fini $L'$ tel que $L'\otimes_{R} R/I\simeq L$.
Par cons\'equent, un $(A\otimes R/I)$-module projectif de type fini
se rel\`eve en un $(A\otimes R)$-module projectif de type fini.

Soit donc $N$ un
$(A\otimes R)$-module projectif et $f:M\to N\otimes_R R/I$ un
morphisme surjectif.
On a $\Ext^1_{A\otimes R}(M,N\otimes_R I)\simeq
\Ext^1_R(\Hom_R(N,R)\otimes_{A\otimes R} M,I)=0$ et
on d\'eduit comme dans la preuve
du lemme \ref{factprojglobal} que $f$ se rel\`eve en un morphisme surjectif
$M\to N$.
\end{proof}

\subsubsection{Invariance stable}
\label{secinvariancestable}
Supposons $A$ et $B$ auto-injectives. Soient
$L$ un $(A\otimes B^\circ)$-module de type fini et 
$L'$ un $(B\otimes A^\circ)$-module de type fini
induisant des \'equivalences stables
inverses entre $A$ et $B$, \ie, tels que
$$L\otimes_B L'\simeq A\oplus\text{ projectif comme }
A^\en\textrm{-modules}$$
$$L'\otimes_A L\simeq B\oplus \text{ projectif comme }
B^\en\textrm{-modules}.$$

\smallskip
Soit $S$ une vari\'et\'e.
Soit $(A^\en\otimes\CO_S)\mstab$ le quotient additif de la
cat\'egorie des $(A^\en\otimes\CO_S)$-modules de type
fini par la sous-cat\'egorie additive des modules localement projectifs.

Soit $\StPic_S(A)$ le groupe des classes d'isomorphisme des 
objets inversibles de $(A^\en\otimes\CO_S)\mstab$.

Les foncteurs $L'\otimes_A-\otimes_A L$ et $L\otimes_B -\otimes_B L'$
induisent des
isomorphismes inverses $\Psi_L:\StPic_A(S)\iso \StPic_B(S)$
et $\Psi_{L'}:\StPic_B(S)\iso \StPic_A(S)$.

\smallskip
Supposons $A$ et $B$ sans modules simples projectifs.
Alors, le morphisme canonique $\Pic_A(S)\to\StPic_A(S)$ est injectif
(on le d\'eduit imm\'ediatement du cas ponctuel).

\begin{lemme}
\label{lemmestable}
Supposons $A$ et $B$ sans modules simples projectifs.
Les isomorphismes $\Psi_L$ et $\Psi_{L'}$
induisent des isomorphismes inverses
$$\xymatrix{\StPic_A(S)/H^1(S,(ZA\otimes\CO_S)^\times)
\ar@<0.5ex>[r]^-{\sim}  & \ar@<0.5ex>[l]^-{\sim}
\StPic_B(S)/H^1(S,(ZB\otimes\CO_S)^\times)}.$$
Ils se restreignent en des isomorphismes inverses
$$\xymatrix{\Pic_A^0(S)/H^1(S,(ZA\otimes\CO_S)^\times)
\ar@<0.5ex>[r]^-{\sim}  & \ar@<0.5ex>[l]^-{\sim}
\Pic_B^0(S)/H^1(S,(ZB\otimes\CO_S)^\times)}.$$
\end{lemme}

\begin{proof}
La premi\`ere assertion est toute aussi imm\'ediate que dans le lemme
\ref{lemmeMorita}.

\smallskip
Soit $S=\Out^0(A)$ et $M\in\CB_A(S)$ correspondant \`a l'injection
$\Out^0(A)\to\Out(A)$.
Soit $N=\Psi_L(M)$. On a $N(1)\simeq B\oplus P$, o\`u $P$ est
un $B^\en$-module projectif de type fini.

Soit $H$ l'ensemble des points ferm\'es $g$ de $S$ tels qu'il existe
un $B^\en$-module $R$ libre de
rang $1$ comme $B$-module et comme $B^\circ$-module avec
$N(g)\simeq R\oplus P$.

D'apr\`es le lemme \ref{factprojglobal}, il existe un voisinage ouvert $\Omega$
de l'identit\'e de $S$ et un $(B^\en\otimes \CO_\Omega)$-module
$T$ tels que $N_{|\Omega}\simeq T\oplus P\otimes\CO_\Omega$.
En particulier, $P$ est facteur direct de $N(g)$
pour tout $g\in \Omega$. 

Puisque $T(1)=B$ est libre de rang $1$ comme $B$-module et comme 
$B^\circ$-module, il existe un voisinage
ouvert $U$ de l'identit\'e dans $\Omega$ tel que
$T_{|U}$ est libre de rang $1$ comme $(B\otimes\CO_U)$-module et comme
$(B^\circ\otimes\CO_U)$-module (lemme \ref{factprojglobal}).
Alors, l'ensemble des points ferm\'es de $U$ est
contenu dans $H$.

Au point g\'en\'erique $\eta$ de $\CO_S$, le module $T_\eta$ ne contient
pas de facteur direct projectif non nul, car l'alg\`ebre $B\otimes \CO_\eta$ 
ne contient pas de facteur direct simple.

Soient $g,h\in S$. Alors, 
$N(g)\otimes_B N(h)\simeq N(gh)\oplus\text{projectif}$.
Prenons $g,h\in H$ et d\'ecomposons $N(g)\simeq R_1\oplus P$ et
$N(h)\simeq R_2\oplus P$. Alors, $N(gh)\simeq R_1\otimes_B R_2\oplus
Q$ o\`u $Q$ est projectif. Puisque $R_1$ et $R_2$ sont libres
de rang $1$ comme $B$-modules et comme
$B^\circ$-modules, on en d\'eduit que $R_1\otimes_B R_2$ est aussi
libre de rang $1$ comme $B$-module et comme $B^\circ$-module.
Dans un voisinage ouvert $V$
de $gh$, on a $N\simeq R\oplus Q\otimes \CO_V$
o\`u $R$ est localement libre de rang $1$ comme $(B\otimes\CO_V)$-module
et comme $(B^\circ\otimes\CO_V)$-module.
Alors, $Q\otimes\CO_\eta\simeq P\otimes\CO_\eta$, donc $Q\simeq P$.
Par cons\'equent, $gh\in H$.

Ainsi, $H$ est un sous-groupe de $S$. Puisqu'il contient les points ferm\'es
d'un ouvert de $S$, c'est $S$ tout entier, car $S$ est connexe.

La proposition \ref{factproj} montre l'existence d'un
$(B^\en\otimes \CO_S)$-module $W$ et d'un
$(B^\en\otimes \CO_S)$-module projectif $Z$ tels que 
$N\simeq W\oplus Z$ et $Z(g)\simeq P$ pour tout $g\in S$. En outre,
$W$ est localement libre de rang $1$ comme $(B\otimes \CO_S)$-module
et comme $(B^\circ\otimes \CO_S)$-module. Par cons\'equent,
$\Psi_L(M)\simeq S\in\Pic_B^f(S)$.
On conclut maintenant comme dans le lemme \ref{lemmederive}.
\end{proof}

\begin{thm}
\label{Out0stable}
Une \'equivalence stable \`a la Morita entre $A$ et $B$ induit un isomorphisme
de groupes alg\'ebriques $\Out^0(A)\iso\Out^0(B)$.
\end{thm}

\begin{proof}
Il suffit de traiter le cas o\`u $A$ et $B$ n'ont pas de modules simples
projectifs.
Le th\'eor\`eme r\'esulte alors du lemme \ref{lemmestable} et du th\'eor\`eme
\ref{repOut}.
\end{proof}

\begin{rem}
Soit $R$ une $k$-alg\`ebre locale noeth\'erienne commutative et $V$ un
$(B\otimes R)$-module, libre de type fini sur $R$, tel que
$V\otimes_R k$ n'a pas de facteur direct projectif. Alors, il
existe un $(A\otimes R)$-module $W$, libre de type fini sur $R$, tel que
$L\otimes_B V\simeq W\oplus P\otimes R$, o\`u $P$ est un $A$-module projectif
et $W\otimes_R k$ n'a pas de facteur direct projectif (lemme \ref{factprojglobal}).
En d'autres termes, si $V'$ est un $B$-module de type fini et 
$L\otimes_B V'\simeq W'\oplus\text{ projectif}$, alors une d\'eformation
de $V'$ s'envoie sur une d\'eformation de $W'$. En particulier, $V'$ est
rigide si et seulement si $W'$ est rigide.
\end{rem}

\subsection{Equivalences d\'eriv\'ees de vari\'et\'es projectives lisses}

\subsubsection{}
Soit $X$ une vari\'et\'e projective lisse sur $k$. On note
$p_1$ et $p_2$ les premi\`eres et deuxi\`emes projections $X\times X\to X$.

Soit $S$ un sch\'ema s\'epar\'e de type fini sur $k$.
On note
$\tilde{\CP}_X(S)$ le groupe des classes d'isomorphisme de
$\CO_{X\times X\times S}$-modules $M$ coh\'erents localement libres sur $S$
tels que $p_{1*}(M(s))$ et 
$p_{2*}(M(s))$ sont 
des fibr\'es en droite num\'eriquement nuls sur $X$, pour tout $s\in S$.
Soit $\CP_X(S)$ le quotient de $\tilde{\CP}_X(S)$ par $\Pic(S)$.

Soit $\CL\in\Pic^0(X\times S)$ et $\sigma:S\to\Aut(X)$. On note
$\Delta_\sigma:X\times S\to X\times X\times S$ le plongement du
graphe relatif de
$\sigma$, \ie, $\Delta_\sigma(x,s)=(x,\sigma(s)(x),s)$.
Alors, $\Delta_{\sigma *}\CL\in\tilde{\CP}_X(S)$ et on obtient ainsi
un morphisme canonique
$$\Pic^0(X\times S)\rtimes\Hom(S,\Aut(X))\to \tilde{\CP}_X(S).$$

\begin{prop}
\label{Picbimod}
On a un isomorphisme canonique de groupes
$$\Pic^0(X/S)\rtimes\Hom(S,\Aut(X))\iso \CP_X(S).$$
En d'autres termes, le foncteur $\CP_X(?)$ est repr\'esent\'e par
$\Pic^0(X)\rtimes \Aut(X)$.
\end{prop}

\begin{proof}
Soit $p_{13}:X\times X\times S\to X\times S$ la projection sur les
premi\`eres et troisi\`emes composantes.
Soit $M\in\tilde{\CP}_X(S)$ tel que
$p_{13*}M\simeq\CO_X\otimes\CO_S$. Alors, il existe un unique
$\sigma:S\to\Aut(X)$ tel que $M\simeq \CO_{\Delta_\sigma(X\times S)}$.
La proposition en d\'ecoule.
\end{proof}

\subsubsection{}
Soit $Y$ une vari\'et\'e projective lisse sur $k$ et
$L\in D^b(X\times Y)$. On pose
$\Phi_L=Rp_{1*}(L\otimes^\BL p_2^*(-)):D^b(Y)\to D^b(X)$ o\`u
$p_1:X\times Y\to X$ et $p_2:X\times Y\to Y$ sont les premi\`eres et
deuxi\`emes projections. Pour $L'\in D^b(Y\times X)$, on pose
$L\boxtimes_Y L'=p_{13*}(p_{12}^*L\otimes^\BL p_{23}^*L')$.

\medskip
Le r\'esultat suivant est pr\'esent\'e dans \cite[Proposition 9.45]{Huy} et
annonc\'e dans \cite[\S 3.2.1]{Rou2}.

\begin{thm}
\label{PicAut}
Une \'equivalence $D^b(Y)\iso D^b(X)$ induit un isomorphisme de groupes
alg\'ebriques
$\Pic^0(Y)\rtimes\Aut^0(Y)\iso \Pic^0(X)\rtimes\Aut^0(X)$.
\end{thm}

\begin{proof}
Il existe $L\in D^b(X\times Y)$ tel que
l'\'equivalence est isomorphe \`a $\Phi_L$ \cite{Or1}.
Soit $L'=R\Hom(L,\CO_{X\times Y})\in D^b(Y\times X)$.
Alors, $L\boxtimes_Y L'\simeq \CO_{\Delta(X)}$, o\`u $\Delta(X)$ est la
diagonale dans $X\times X$. Soit $S=\Pic^0(Y)\rtimes\Aut^0(Y)$ et soit
$N\in \CP_Y(S)$ correspondant \`a l'identit\'e via la proposition
\ref{Picbimod}. Soit $N=L\boxtimes_Y M\boxtimes_Y L'\in D^b(X\times X\times S)$.
Alors, la fibre \`a l'identit\'e de $N$ est isomorphe \`a
$\CO_{\Delta(X)}$. Par cons\'equent, il existe un voisinage ouvert
$\Omega$ de l'identit\'e dans $S$ tel que $N_{|\Omega}$ a ses faisceaux
de cohomologie nuls en dehors du degr\'e $0$. Quitte \`a r\'etr\'ecir
$\Omega$, on a $N_{|\Omega}\in\CP_X(\Omega)$ et on conclut comme dans la
preuve du lemme \ref{lemmederive}, via la proposition \ref{Picbimod}.
\end{proof}

\begin{rem}
Dans le cas o\`u $X$ et $Y$ sont des vari\'et\'es ab\'eliennes, on retrouve
un r\'esultat d'Orlov \cite{Or2}:
une \'equivalence $D^b(Y)\simeq D^b(X)$
induit un isomorphisme $Y\times \hat{Y}\iso X\times\hat{X}$.
\end{rem}

\begin{rem}
Le lecteur int\'eress\'e pourra formuler une g\'en\'eralisation des
th\'eor\`emes \ref{invarderiv} et \ref{PicAut} au cas d'une
$\CO_X$-alg\`ebre finie, o\`u $X$ est une vari\'et\'e projective lisse.
\end{rem}

\newpage
\section{Alg\`ebres gradu\'ees}

\subsection{G\'en\'eralit\'es}
\subsubsection{Dictionnaire}
\label{dictionnaire}
Soit $A$ une $k$-alg\`ebre de dimension finie.
Une graduation sur $A$ correspond \`a la donn\'ee d'une action alg\'ebrique
de $\BG_m$ sur $A$, {\em i.e.}, \`a un morphisme de groupe alg\'ebriques
$\pi:\BG_m\to\Aut(A)$.
Pour $A=\bigoplus_{i\in\BZ} A_i$, alors $x\in k^*$ agit sur $A_i$
par multiplication par $x^i$.

Cela correspond \`a son tour \`a la donn\'ee
d'un morphisme d'alg\`ebres $\rho:A\to A[t,t^{-1}]$.
On a $\rho(a)=at^i$ pour $a\in A_i$.
Ce morphisme
v\'erifie les propri\'et\'es
suivantes~:
\begin{itemize}
\item (inversibilit\'e) Le morphisme $\rho\otimes 1:A[t,t^{-1}]\to A[t,t^{-1}]$
est un isomorphisme, c'est-\`a-dire, la composition
$\rho_x:A\Rarr{\rho} A[t,t^{-1}]\to A[t,t^{-1}]/(t-x)=A$ est un isomorphisme
pour tout $x\in k^\times$~;
\item (multiplicativit\'e) Le diagramme suivant est commutatif
$$\xymatrix{
A \ar[rr]^\rho \ar[d]_\rho && A\otimes k[t,t^{-1}]\ar[d]^{1\otimes \mu} \\
A\otimes k[t,t^{-1}]\ar[rr]_{\rho\otimes 1} &&
  A\otimes k[t,t^{-1}]\otimes k[t,t^{-1}]
}$$
o\`u $\mu:k[t,t^{-1}]\to k[t,t^{-1}]\otimes k[t,t^{-1}],\
t\mapsto t\otimes t$ est la comultiplication.
\end{itemize}

On obtient alors un $(A\otimes A^\circ)[t,t^{-1}]$-module $X$ \`a partir
d'une graduation.
On a $X=A[t,t^{-1}]$ comme $A[t,t^{-1}]$-module et $A$ agit
\`a droite par multiplication pr\'ec\'ed\'ee de $\rho$~:
pour $x\in X$ et $a\in A_i$, alors $x\cdot a=t^i xa$.

\subsubsection{Rel\`evement d'actions de tores}
Soit $A$ une $k$-alg\`ebre de dimension finie et
$G$ un tore sur $k$ agissant sur $A$.

La proposition suivante est classique.
\begin{prop}
\label{graduable}
Soit $M$ un $A$-module de type fini
tel que $A_g\otimes_A M\simeq M$ pour tout
$g\in G$. Alors, $M$ s'\'etend en un $(A\rtimes G)$-module.
Si $M$ est ind\'ecomposable, alors cette extension est unique \`a
multiplication par un caract\`ere de $G$ pr\`es et \`a isomorphisme
pr\`es.
\end{prop}

\begin{proof}
Soit $A'$ l'image de $A$ dans $\End_k(M)$.
Soit $\tilde{G}$ le sous-groupe de $G\times \End_k(M)^\times$ form\'e des
paires d'\'el\'ements qui ont la m\^eme action sur $A'$ (\ie, les
$(g,f)$ tels que $\rho(g(a))=f\rho(a)f^{-1}$ pour tout $a\in A$, o\`u
$\rho:A\to \End_k(M)$ est le morphisme structurel).

Une extension de $M$ en un $(A\rtimes G)$-module correspond en la donn\'ee
d'une scission de la premi\`ere projection $\tilde{G}\to G$ (qui est surjective
par hypoth\`ese)~: l'action de
$G$ est alors donn\'ee en composant avec la seconde projection
$\tilde{G}\to\End_k(M)^\times$.

On a un diagramme commutatif dont les lignes et colonnes sont exactes
$$\xymatrix{
        &          & 1\ar[d]                & 1\ar[d] \\
1\ar[r] & U \ar[r]\ar@{=}[d] & \End_A(M)^\times\ar[r]\ar[d] & L\ar[r]\ar[d] & 1 \\
1\ar[r] & U \ar[r] & \tilde{G}\ar[r]\ar[d] & \hat{G}\ar[r]\ar[d] & 1 \\
        &          & G \ar@{=}[r]\ar[d] & G\ar[d] \\
        &          & 1        & 1 \\
}$$
o\`u $U=1+J(\End_A(M))$ est le radical unipotent de $\End_A(M)^\times$
et $L$ un produit de groupes lin\'eaires.

L'extension de $G$ par $L$ est scind\'ee, car $G$ est un tore. Fixons une
scission.
L'extension correspondante
du tore $G$ par le groupe unipotent $U$ est scind\'ee.
Par cons\'equent, l'extension de $G$ par $\End_A(M)^\times$ est
scind\'ee, donc $M$ s'\'etend en un $(A\rtimes G)$-module.

Passons \`a l'unicit\'e. Pour $M$ ind\'ecomposable, on a
$L\simeq\BG_m$. La scission de l'extension de $G$ par $U$ est
unique \`a conjugaison par $U$ pr\`es,
donc la scission de l'extension de $G$ par $\End_A(M)^\times$
est unique \`a multiplication par un morphisme de $G$ dans
$k^\times\cdot 1_M$ pr\`es et \`a conjugaison par $U$ pr\`es.
\end{proof}

Notons au passage que si $M$ est un $(A\rtimes G)$-module ind\'ecomposable,
alors $M$ reste ind\'ecomposable comme $A$-module.

\subsection{Puissances cycliques de l'espace cotangent}

\subsubsection{}
Soit $A$ une $k$-alg\`ebre de dimension
finie. Soit $T=(JA/J^2A)^*$ l'espace tangent.
C'est un $(A/JA,A/JA)$-bimodule.

Soit $\CS$ un ensemble de repr\'esentants des classes d'isomorphisme
de $A$-modules simples.
On a un isomorphisme canonique de $(A/JA,A/JA)$-bimodules
$$f_0:A/JA\iso \bigoplus_{V\in\CS} \Hom_k(V,V)$$
$$a\mapsto \sum_V (v\in V\mapsto av).$$

Soit $W$ un $A$-module simple.
Le morphisme canonique
$$\Hom_A(JA,W)\iso \Hom_A(T^*,W)$$
est un isomorphisme.
La suite exacte $0\to JA\to A\to A/JA\to 0$ fournit un isomorphisme
de $A$-modules
$$\Hom_A(JA,W)\iso \Ext^1_A(A/JA,W).$$
On a un isomorphisme de $A$-modules induit par $f_0$
$$\Ext^1_A(A/JA,W)\iso
 \bigoplus_{V\in\CS}\Ext^1_A(V,W)\otimes V.$$
Enfin, on a un isomorphisme canonique de $A^{\en}$-modules
$$T^*\iso \bigoplus_{W\in\CS} \Hom_k(\Hom_A(T^*,W),W).$$

Il ne nous reste plus qu'\`a composer tous ces isomorphismes et \`a dualiser~:

\begin{lemme}
\label{descriptionT}
On a un isomorphisme canonique de $A^\en$-modules
$$T\iso \bigoplus_{V,W\in\CS}\Hom_k(V,W)\otimes \Ext^1_A(W,V).$$
\end{lemme}

\subsubsection{}
On note $T^{\odot_A^n}=T\otimes_A \cdots\otimes_A T\otimes_A$ le produit
tensoriel cyclique $n$-\`eme de $T$ au-dessus de $A$.
L'action naturelle de $\Aut(A)$ sur $T$ fournit une action diagonale sur
$T^{\odot_A^n}$.

Le lemme suivant est imm\'ediat.

\begin{lemme}
L'action de $\Aut(A)$ sur $T^{\odot_A^n}$
se factorise en une action de $\Out(A)$.
\end{lemme}

L'isomorphisme canonique du lemme \ref{descriptionT} fournit~:

\begin{prop}
\label{tenseurcyclique}
On a un isomorphisme canonique
$$T^{\odot_A^n}\iso\bigoplus_{V_1,\ldots,V_n\in\CS}
 \Ext^1_A(V_1,V_2)\otimes\cdots\otimes\Ext^1_A(V_{n-1},V_n)\otimes
\Ext^1_A(V_n,V_1).$$
\end{prop}

\subsection{Changement de graduation}
\label{se:chgtgrad}

\begin{lemme}
Tout morphisme $\BG_m\to\Out(A)$ se rel\`eve en un morphisme
$\BG_m\to\Aut(A)$.
\end{lemme}

\begin{proof}
L'image d'un tore maximal $T$ de $\Aut^0(A)$ par le morphisme canonique
$\phi:\Aut^0(A)\to\Out^0(A)$
est un tore maximal de $\Out^0(A)$. Puisque le noyau
$\Int(A)$ de $\phi$ est connexe, la restriction \`a $T$ de $\phi$
est scind\'ee. Le lemme d\'ecoule maintenant du fait que tout morphisme
$\BG_m\to\Out(A)$ est conjugu\'e \`a un morphisme d'image contenue dans
$\phi(T)$.
\end{proof}

De m\^eme, on a

\begin{lemme}
Tout morphisme $\BG_m\to \Int(A)$ se rel\`eve en un morphisme $\BG_m\to
A^\times$.
\end{lemme}

\begin{prop}
\label{conjugaison}
Deux morphismes $\pi,\pi':\BG_m\to\Aut(A)$ sont conjugu\'es
si et seulement si les alg\`ebres gradu\'ees $(A,\pi)$ et $(A,\pi')$
sont isomorphes.
\end{prop}

\begin{proof}
Soit $\alpha\in\Aut(A)$ tel que $\pi'=\alpha \pi \alpha^{-1}$. Alors,
l'automorphisme $\alpha$ de $A$ induit un isomorphisme
de $(A,\pi)$ vers $(A,\pi')$.

R\'eciproquement, soit $\alpha\in\Aut(A)$ induisant un isomorphisme
de $(A,\pi)$ vers $(A,\pi')$. Alors, $\pi'=\alpha \pi \alpha^{-1}$.
\end{proof}

\begin{prop}
Supposons que l'alg\`ebre $A$ est basique.
Soient $\pi,\pi':\BG_m\to\Aut(A)$. Les assertions suivantes sont \'equivalentes~:
\begin{itemize}
\item[(i)]
$\pi$ et $\pi'$ induisent des morphismes conjugu\'es $\BG_m\to\Out(A)$
\item[(ii)]
il existe une famille d'entiers $\{d_i(V)\}$ o\`u $V$ d\'ecrit
les $A$-modules simples et $1\le i\le \dim V$ tels que
l'alg\`ebre gradu\'ee $(A,\pi')$ est isomorphe \`a 
$\Endgr_{(A,\pi)}(\bigoplus_{V,i}P_V\langle d_i(V)\rangle)$.
\end{itemize}
\end{prop}

\begin{proof}
Soit $S$ une sous-alg\`ebre semi-simple maximale de $A$, $S'$ une
sous-alg\`ebre de Cartan de $S$ (un produit maximal de corps dont chacune des
unit\'es est un idempotent primitif de $S$) et $T'={S'}^\times$ (c'est
un tore maximal de $S^\times$). Alors, $T'$ est un tore maximal de $A^\times$.
Soit $T''$ son
image dans $\Int(A)$. Soit $T$ un tore maximal de $\Aut(A)$
contenant $T''$.

Quitte \`a conjuguer $\pi$ et $\pi'$ par des \'el\'ements
de $\Aut(A)$, ce qui ne change pas la classe d'isomorphisme des alg\`ebres
gradu\'ees (proposition \ref{conjugaison}), on peut supposer que $\pi$ et $\pi'$
sont \`a valeur dans $T$.
Supposons que $\pi$ et $\pi'$ deviennent conjugu\'es dans $\Out(A)$. On
peut supposer, quitte \`a conjuguer $\pi$, que $\pi$ et $\pi'$
co\"{\i}ncident dans $\Out(A)$ et alors, $\pi'\pi^{-1}$ est un
cocaract\`ere $\psi$ de $T''$, que l'on peut relever en $\phi:\BG_m\to T'$.

\smallskip
On a $S=\prod_V e_VS$ o\`u $V$ d\'ecrit l'ensemble des classes
d'isomorphisme de $A$-modules simples
et $e_V$ est l'idempotent primitif du centre de $S$ qui n'agit pas par $0$
sur $V$.
D\'ecomposons $e_V$ en sommes d'idempotents primitifs de $S'$, $e_{V,1},\ldots,
e_{V,n_V}$ deux \`a deux orthogonaux, $e_V=\sideset{}{^\perp}\sum_{\!\!i} e_{V,i}$.
Il existe des entiers $d_i(V)$ tels que
$\phi(\alpha)=\sum_{V,i}\alpha^{d_i(V)}e_{V,i}$ pour tout
$\alpha\in k^\times$.

Consid\'erons maintenant le cocaract\`ere $\psi:\BG_m\to\Int(A),
\alpha\mapsto (a\mapsto \phi(\alpha)a\phi(\alpha)^{-1})$.
Alors, $\BG_m$ agit (via $\psi$) avec le poids $d_i(V)-d_j(W)$ sur
$e_{V,i}Ae_{W,j}$.

Par cons\'equent, l'alg\`ebre gradu\'ee $(A,\psi \pi)$ est \'egale
\`a l'alg\`ebre gradu\'ee
$\bigoplus_{V,W,i,j}e_{V,i}Ae_{W,j}\langle d_j(W)-d_i(V)\rangle\iso
\Endgr_{(A,\pi)}(\bigoplus_{V,i}Ae_{V,i}\langle d_i(V)\rangle)$.
La r\'eciproque est claire.
\end{proof}

Ainsi, la donn\'ee d'un morphisme non trivial $\BG_m\to\Out(A)$ d\'etermine
une graduation ``\`a \'equivalence de Morita pr\`es''~:

\begin{cor}
Supposons $A$ basique.
Deux graduations sur $A$ donnent lieu \`a des alg\`ebres gradu\'ees
Morita-\'equivalentes si et seulement si les cocaract\`eres correspondant
de $\Out(A)$ sont conjugu\'es.
\end{cor}

\begin{rem}
Lorsque l'alg\`ebre $A$ n'est pas basique, les r\'esultats pr\'ec\'edents
restent vrais en rempla\c{c}ant la conjugaison par $\Out(A)$ par la
conjugaison par $\Pic(A)$.

Notons que l'ensemble des classes de ``Morita-\'equivalence'' de graduations
n'est pas invariant par \'equivalence d\'eriv\'ee. Soit $A$ une
$k$-alg\`ebre de dimension finie munie d'une graduation qui n'est 
pas Morita-\'equivalente \`a la graduation triviale. Soit $A'$ une
alg\`ebre d\'eriv\'ee-\'equivalente \`a $A$, mais non Morita-\'equivalente.
On munit $A'$ d'une graduation compatible \`a l'\'equivalence.
On peut munir $A\times A'$ de deux graduations non Morita-\'equivalentes:
la premi\`ere est la graduation non triviale sur $A$ et la
graduation triviale sur $A'$, 
la seconde est la graduation triviale sur $A$ et la
graduation non triviale sur $A'$.

  L'alg\`ebre $A\times A'$ est
d\'eriv\'ee-\'equivalente \`a $A\times A$ et les deux graduations
pr\'ec\'edentes sont compatibles respectivement avec les deux graduations
suivantes sur $A\times A$:
la premi\`ere est la graduation non triviale sur le premier facteur $A$ et la
graduation triviale sur le second facteur $A$, 
la seconde est la graduation triviale sur le premier facteur $A$ et la
graduation non triviale sur le second facteur $A$.
Ces deux graduations sur $A\times A$ sont Morita-\'equivalentes.
\end{rem}

\subsection{Graduations et positivit\'e}

\subsubsection{}
Soit $A$ une $k$-alg\`ebre finie gradu\'ee.
Le lemme suivant est imm\'ediat (on reprend les notations de
\S \ref{dictionnaire})

\begin{lemme}
\label{degpositifs}
Les assertions suivantes sont \'equivalentes
\begin{itemize}
\item
la graduation est en degr\'es positifs ;
\item
le morphisme $\pi$ s'\'etend en un morphisme de vari\'et\'es $\BA^1\to\End(A)$
(o\`u $\BG_m$ est vu comme l'ouvert $\BA^1-\{0\}$) ;
\item
le morphisme $\rho$ provient d'un morphisme d'alg\`ebres $A\to A[t]$.
\end{itemize}
\end{lemme}

\smallskip
\begin{rem}
Lorsque $A$ est basique (\ie, les $A$-modules simples sont de dimension $1$),
alors $A/JA$ est un produit de corps, donc est en degr\'e $0$.
Notons que si $A$ n'est pas basique, il se peut que $A/JA$ ne soit
pas en degr\'e $0$. En changeant la graduation sans
changer le morphisme $\BG_m\to\Out(A)$, on peut se ramener au cas o\`u
$A/JA$ est en degr\'e $0$.

Lorsque $A$ est semi-simple et en degr\'es positifs, alors $A=A_0$.
\end{rem}

\begin{lemme}
\label{Jacobson}
L'alg\`ebre $A$ est en degr\'es positifs si et seulement si $A/J^2A$
est en degr\'es positifs.
\end{lemme}

\begin{proof}
Supposons $A/J^2A$ en degr\'es positifs.
Soit $S$ une sous-alg\`ebre de $A_0$ qui s'envoie bijectivement sur $A/JA$.
Une famille d'\'el\'ements de $JA$ qui engendre $JA/J^2A$ comme
$k$-espace vectoriel engendre $A$ comme $S$-alg\`ebre, d'o\`u le
r\'esultat.
\end{proof}

\subsubsection{}

\begin{prop}
\label{caracpositivite}
Soit $\bar{\pi}:\BG_m\to\Out(A)$. Les assertions suivantes sont
\'equivalentes
\begin{itemize}
\item[(i)] il existe un rel\`evement de $\bar{\pi}$ en un morphisme
$\BG_m\to\Aut(A)$ qui s'\'etend en un morphisme
$\BA^1\to \End(A)$ ;
\item[(ii)] pour tout entier positif $n$, le groupe
$\BG_m$ agit (via $\bar{\pi}$) avec
des poids positifs sur $(JA/J^2A)^{\odot_A^n}$ ;
\item[(iii)] pour toute graduation associ\'ee \`a un morphisme
$\pi:\BG_m\to\Aut(A)$ relevant $\bar{\pi}$,
pour toute suite $S_n=S_0,S_1,\ldots,S_{n-1}$ de $A$-modules
simples de degr\'e $0$ et pour tous $d_0,\ldots,d_{n-1}\in\BZ$ tels
que $\Ext^1(S_i,S_{i+1}\langle -d_i\rangle)\not=0$, alors $\sum_i d_i\ge 0$ ;
\item[(iv)] le morphisme $\BG_m\to\Out(A/J^2A)$ d\'eduit de $\bar{\pi}$
par le morphisme canonique se rel\`eve en un morphisme
$\BG_m\to\Aut(A/J^2A)$ qui s'\'etend en un morphisme $\BA^1\to \End(A/J^2A)$.
\end{itemize}
\end{prop}

\begin{proof}
L'\'equivalence entre (ii) et (iii) r\'esulte de la proposition
\ref{tenseurcyclique}. Notons que (i) implique bien s\^ur (ii) et (iv). 

Soit $f:\Out(A)\to\Out(A/J^2A)$ le morphisme canonique.
Le morphisme canonique $\Int(A)\to\Int(A/J^2A)$ est surjectif et sa restriction
\`a un tore maximal est un isomorphisme, donc
tout rel\`evement de $f\bar{\pi}$ en un morphisme 
$\BG_m\to\Aut(A/J^2A)$ provient d'un morphisme $\pi:\BG_m\to\Aut(A)$
relevant $\bar{\pi}$. Les lemmes \ref{degpositifs} et \ref{Jacobson}
montrent alors que (iv)$\Longrightarrow$(i).

Supposons (iii).
Choisissons un rel\`evement de $\bar{\pi}$ en un morphisme 
$\pi:\BG_m\to\Aut(A)$ tel que $A/JA$ est en degr\'e $0$.

Soient $S$ et $T$ deux $A$-modules simples de degr\'e $0$. Soit
$f(S,T)$ le plus petit entier $d$ tel qu'il existe des
$A$-modules simples $S_1,\ldots,S_n$ de degr\'e $0$
 et des entiers $d_0,\ldots,d_n$ v\'erifiant
\begin{multline*}
\Ext^1(S,S_1\langle -d_0\rangle)\not=0,\Ext^1(S_1,S_2\langle
-d_1\rangle)\not=0,\ldots,\\
\Ext^1(S_{n-1},S_n\langle -d_{n-1}\rangle)\not=0,\Ext^1(S_n,T\langle
-d_n\rangle)\not=0
\end{multline*}
$$\textrm{et }\sum d_i=d.$$
On a $f(S,T)+f(T,U)\ge f(S,U)$ et
$f(S,T)+f(T,S)\ge 0$ pour tous $S,T$ et $U$ simples
de degr\'e $0$.
Par cons\'equent, il existe une application $d$ de l'ensemble des classes
d'isomorphisme de
$A$-modules simples de degr\'e $0$ vers les entiers telle que la fonction
$f'(S,T)=f(S,T)+d(S)-d(T)$ ne prend que des valeurs positives
(cf lemme \ref{lemmefonction} ci-dessous).
En rempla\c{c}ant $A$ par l'alg\`ebre gradu\'ee
$\Endgr_{(A,\pi)}(\bigoplus_S P_S^{\dim S}\langle d(S)\rangle)$,
on a $\Ext^1(S,T\langle d\rangle)=0$ pour $d>0$ pour
tous modules simples $S$ et $T$ de degr\'e $0$. Puisque on a un isomorphisme
de $(A,A)$-bimodules gradu\'es (lemme \ref{descriptionT})
$$JA/J^2A\simeq \bigoplus_{S,T,d}
(T\otimes S^*)^{\dim\Ext^1(S,T\langle d\rangle)}
\langle d\rangle,$$
on en d\'eduit que $JA/J^2A$ est en degr\'es positifs, donc
$A$ aussi, par le lemme \ref{Jacobson}. On a donc d\'emontr\'e
(iii)$\Longrightarrow$(i).
\end{proof}

\begin{lemme}
\label{lemmefonction}
Soit $E$ un ensemble fini et $f:E\times E\to\BZ$ telle que
$f(x,y)+f(y,z)\ge f(x,z)$ et $f(x,y)+f(y,x)\ge 0$ pour tous $x,y,z\in E$.
Alors, il existe $d:E\to\BZ$ telle que $f':E\times E\to\BZ$ donn\'ee
par $f'(x,y)=f(x,y)+d(x)-d(y)$ ne prend que des valeurs positives.
\end{lemme}

\begin{proof}
On prouve le lemme par r\'ecurrence sur la valeur absolue de la somme
des valeurs n\'egatives de $f$.
Soit $E_-$ l'ensemble des $x\in E$ tels qu'il existe $y\in E$ avec $f(x,y)<0$.
Supposons $E_-\not=\emptyset$.
Soit $c:E\to\BZ$ donn\'ee par $c(x)=1$ si $x\in E_-$ et
$c(x)=0$ sinon. On consid\`ere la fonction $g$ donn\'ee par
$g(x,y)=f(x,y)+c(x)-c(y)$.
Si $x\in E_-$, alors $g(x,y)=f(x,y)$ ou $g(x,y)=f(x,y)+1$.
Soient $x\not\in E_-$ et $y\in E$ tels que $f(x,y)=0$. Alors,
tout $z$, on a $f(x,z)\le f(y,z)$. Par cons\'equent, $y\not\in E_-$.
On en d\'eduit que $g(x,y)=f(x,y)$. On a montr\'e que la somme
des valeurs absolues des valeurs n\'egatives de $g$ est strictement
inf\'erieure \`a la somme correspondante pour $f$.
Puisque $g$ v\'erifie encore les hypoth\`eses du lemme, on conclut par
r\'ecurrence.
\end{proof}

\begin{rem}
Dans la proposition, on peut bien s\^ur supposer que $n$ est au plus
le nombre de $A$-modules simples.
\end{rem}

\begin{rem}
La positivit\'e de l'homologie cyclique ou de l'homologie de Hochschild
de $A$ ne suffit pas pour avoir une graduation positive.
\end{rem}

\subsection{Dualit\'e}
\label{dualite}
\subsubsection{}
Soit $A$ une alg\`ebre de Frobenius, \ie, on se donne
$\nu$ un automorphisme de $A$ et
$A^*\iso A_\nu$ un isomorphisme de $(A,A)$-bimodules.
On suppose $A$ gradu\'ee et ind\'ecomposable.
On notera $n=n_A=\sup\{i|A_i\not=0\}-\inf\{i|A_i\not=0\}$.

\begin{prop}
\label{dual}
On a des isomorphismes de
$(A,A)$-bimodules gradu\'es
$A^*\simeq A_\nu \langle n\rangle$. Plus g\'en\'eralement,
$(\soc^i A)^*\simeq (A/J^i A)\otimes_A A_\nu\langle n\rangle$ pour tout
$i\ge 0$.
\end{prop}

\begin{proof}
La premi\`ere assertion r\'esulte imm\'ediatement du lemme \ref{graduable},
puisque $A_\nu$ et $A^*$ sont ind\'ecomposables et
isomorphes comme $(A,A)$-bimodules non gradu\'es.

L'isomorphisme canonique $(\soc^i A)^*\iso (A^*/J^i A^*)$, permet de
d\'eduire la deuxi\`eme assertion.
\end{proof}

Un isomorphisme $f:A_\nu\langle n\rangle\iso A^*$ provient
d'une forme lin\'eaire
$t:A\langle n\rangle \to k$~: on a
$$f=\hat{t}:A_\nu\langle n\rangle \iso A^*,\ a\mapsto (a'\mapsto t(a'a)).$$
Rappelons que l'alg\`ebre $A$ (munie de $f$ ou, de mani\`ere \'equivalente,
de $t$) est sym\'etrique si $\nu=\id$.

\smallskip
Fixons une telle forme. Pour tout $A$-module gradu\'e $V$,
on dispose alors d'isomorphismes de $A^\circ$-modules gradu\'es (le premier
est une adjonction, le second fournit par $f$)~:
$$V^*\iso\Homgr_A(V,A^*)\iso\Homgr_A(V,A_\nu)\langle n\rangle.$$

Soit $X$ un $A$-module gradu\'e et $P$ un $A$-module projectif ind\'ecomposable
gradu\'e.
On a un accouplement parfait
$$\Hom(P,X)\otimes_k \Hom(X,P_\nu\langle n \rangle)\to k.$$

\subsubsection{}
\label{dualiteautoinjective}
Supposons maintenant $A$ non-n\'ecessairement de Frobenius, mais auto-injective.
Soit $A'$ une
alg\`ebre basique Morita-\'equivalente \`a $A$~: c'est encore une
alg\`ebre auto-injective, donc elle est de Frobenius.

On dispose d'un
automorphisme $\nu$ de l'ensemble des classes d'isomorphisme
de $A$-modules simples (=ensemble des classes d'isomorphisme
de $A'$-modules simples),
l'automorphisme de Nakayama, et d'un accouplement parfait
$$\Hom(P_V,X)\otimes_k \Hom(X,P_{\nu(V)}\langle n \rangle)\to k$$
pour tout $A$-module simple $V$.

\subsection{Matrices de Cartan}

On suppose ici $A$ auto-injective et ind\'ecomposable.
Soit $\CS$ l'ensemble des classes d'isomorphisme de $A$-modules simples de
degr\'e $0$ et $r$ son cardinal.
La matrice de Cartan gradu\'ee $C$ de $A$ est la $(\CS\times\CS)$-matrice sur
$\BZ[q,q^{-1}]$ dont le coefficient $C(V,W)$ est
$$C(V,W)=\sum_i q^i\dim \Hom(P_V,P_W\langle i\rangle).$$

On notera $x\mapsto \bar{x}$ l'automorphisme de $\BZ[q,q^{-1}]$
qui \'echange $q$ et $q^{-1}$.

\begin{prop}
\label{Cartan}
On a $C(V,W)=q^n \overline{C(W,\nu(V))}$.

Si $A_{<0}=0$, alors
$$\det C=c+a_1q+\ldots+a_{nr-1}q^{nr-1}+\eps(\nu)cq^{nr}$$
o\`u $a_1,\ldots,a_{nr-1}$ sont des entiers,
$c$ est le d\'eterminant de la matrice de Cartan de $A_0$ et
$\eps(\nu)$ est le signe de la permutation $\nu$ de $\CS$.
\end{prop}

\begin{proof}
La premi\`ere \'egalit\'e r\'esulte de l'isomorphisme
du \S \ref{dualiteautoinjective}:
$$\Hom(P_V,P_W\langle i\rangle)^*\simeq
\Hom(P_W,P_{\nu(V)}\langle n-i\rangle).$$

Lorsque $A_{<0}=0$, la matrice de Cartan est \`a coefficients dans
$\BZ[q]$ et il est clair que sa partie constante est la
matrice de Cartan de $A_0$. Le degr\'e de $\det C$ est au plus
$nr$ et le coefficient de $q^{nr}$ dans $\det C$ est le terme constant de
$q^{nr}\det\bar{C}=\eps(\nu)\det C$.
\end{proof}

\subsection{Degr\'e $0$}
\label{degre0}

Lorsque $A/JA$ est concentr\'ee en degr\'e $0$ (c'est le cas si
$A$ est basique car alors $A/J(A)$ est un produit de copies de $k$),
alors
les $A_0$-modules simples sont les restrictions des $A$-modules simples
et $A_{<0}\oplus A_{>0}\subseteq J(A)$.

\smallskip
On prend dans le reste de cette partie \S \ref{degre0} pour $A$
une alg\`ebre gradu\'ee telle que $A_{<0}=0$.
On a alors $A_0\iso A/A_{>0}$.

On a un foncteur exact pleinement fid\`ele
de $A/A_{>0}\otimes_{A_0}-:A_0\mMod\to A\mmodgr$
d'image la sous-cat\'egorie de Serre
des $A$-modules gradu\'es concentr\'es en degr\'e $0$.
C'est un foncteur exact, d'inverse \`a gauche le foncteur exact ``poids 0''.
Le foncteur $\Hom_A(A/A_{>0},-)$ est adjoint \`a droite de 
$A/A_{>0}\otimes_{A_0}-$ et est aussi un inverse \`a gauche.

En d\'erivant, on obtient
\begin{prop}
\label{pr:A0ff}
Le foncteur $A/A_{>0}\otimes_{A_0}^\BL-:
D(A_0\mMOD)\to D(A\mMODgr)$ est pleinement fid\`ele.
\end{prop}

\begin{proof}
En effet, le foncteur $R\Hom_{A\mMODgr}^\bullet(A/A_{>0},-)$ est adjoint \`a
droite. C'est aussi un inverse \`a gauche car $\Ext^i_A(A/A_{>0},M)=0$
pour $i>0$ et $M$ un $A$-module tel que $M_0=M$. On en d\'eduit que
$A/A_{>0}\otimes_{A_0}^\BL-$ est pleinement fid\`ele.
\end{proof}

\begin{prop}
\label{dimproj}
On a l'\'equivalence des assertions suivantes~:
\begin{itemize}
\item[(i)] L'alg\`ebre $A_0$ est de dimension globale finie.
\item[(ii)] Le $A$-module $A_0$ engendre $D^b(A\mModgr)$ comme sous-cat\'egorie
\'epaisse close par d\'ecalage $\langle -\rangle$.
\item[(iii)] Pour tous $V,W$ dans $D^b(A\mModgr)$, on a
$\Hom_A(V,W[i])=0$ pour $i>>0$.

\smallskip
Si en outre $A$ est auto-injective, ces assertions sont \'equivalentes \`a
\item[(iv)]
Pour $V$ et $W$ dans $A\mstabgr$, on a
$\Ext^i_A(V,W)=0$ pour $i>>0$.
\end{itemize}

\smallskip
Plus pr\'ecis\'ement, la dimension globale de $A_0$ est le
plus petit entier $d$ tel que pour tous $A$-modules $V$ et $W$,
on a $\Ext^i_A(V,W)=0$ pour $i\ge (d+1)(1+b-a)$, o\`u $a$ (resp. $b$) est le
plus petit (resp. grand) entier tel que $V_a\not=0$ (resp. $W_b\not=0$).
\end{prop}

\begin{proof}
L'\'equivalence de (iii) et (iv) est claire.

(i) est \'equivalent \`a la propri\'et\'e que les $A_0$-modules simples
sont quasi-isomorphes \`a des complexes born\'es de $A_0$-module projectifs.
Par cons\'equent, (i) implique (ii).

Supposons maintenant (ii). Soit $V$ un $A_0$-module. Alors, $V$ est
extension it\'er\'ee de complexes $P[i]\langle j\rangle$ o\`u
$P$ est un $A_0$-module projectif. Apr\`es
application du foncteur ``poids $0$'', on obtient $V$ comme
extension it\'er\'ee de complexes $P[i]$. Puisque
$\Ext^i_A(A_0,A_0)=\Ext^i_{A_0}(A_0,A_0)=0$ pour $i\not=0$, on en
d\'eduit que $V$ est quasi-isomorphe \`a un complexe born\'e
de $A_0$-modules projectifs. On a ainsi montr\'e (i).

\smallskip
Supposons $\Ext^i_A(V,W)=0$ pour $i\ge n$ et pour tous
$A_0$-modules simples $V$ et $W$. Alors, $A_0$ est de dimension
globale inf\'erieure \`a $n$, puisque
$\Ext^i_{A_0}(V,W)=\Ext^i_A(V,W)$. En particulier, (iii) implique (i).

\smallskip
Supposons $A_0$ de dimension globale $d$. Nous allons
montrer que pour tout $A_0$-module $V$, le module $\Omega_A^{d+1} V$ est
engendr\'e en degr\'es $\ge 1$.

Le $A$-module $V$ est quasi-isomorphe \`a un complexe born\'e $C$ de
$A_0$-modules projectifs avec $C^i=0$ pour $i<-d$ et $i>0$.
Chaque $A$-module $C^i$ admet une r\'esolution projective dont tous
les termes, sauf le premier, sont engendr\'es en degr\'es $\ge 1$.
On en  d\'eduit que $V$ a une r\'esolution projective o\`u tous les
termes sont engendr\'es en degr\'es $\ge 1$, sauf les $d+1$ premiers.
Par cons\'equent, $\Omega_A^{d+1} V$ est engendr\'e en degr\'es $\ge 1$.

On en d\'eduit que si $S$ et $T$ sont deux $A_0$-modules simples, alors
$\Ext^i_A(S,T\langle n\rangle)=0$ pour $i\ge (d+1)(1-n)$.
On obtient alors que (i) implique (iii) et on d\'eduit aussi
la derni\`ere partie de la proposition.
\end{proof}

\section{Equivalence stables gradu\'ees}

\subsection{Invariances}

\subsubsection{Graduation}
\label{graduation}

Soit $A$ une $k$-alg\`ebre auto-injective gradu\'ee. Soit $B$
une $k$-alg\`ebre auto-injective.

\begin{thm}
\label{th:gradstable}
Soient $L$ un
$(A,B)$-bimodule et $L'$ un $(B,A)$-bimodule induisant des \'equivalences
stables (non gradu\'ees) inverses entre $A$ et $B$ (cf
\S \ref{secinvariancestable}).

Alors, il existe une graduation sur $B$ et des structures de
bimodules gradu\'es sur $L$ et $L'$. 
En particulier, $L$ et $L'$ induisent des \'equivalences inverses entre
les cat\'egories $B\mstabgr$ et $A\mstabgr$.
\end{thm}

\begin{proof}
Il suffit de traiter le cas o\`u $L$ et $L'$ sont ind\'ecomposables, quitte
\`a remplacer $A$ et $B$ par des blocs.

Soit $\sigma:\BG_m\to\Aut(A)$ induisant la graduation sur $A$.
Soit $\tau$ un morphisme $\BG_m\to\Aut(B)$ dont l'image dans
$\Hom(\BG_m,\Out^0(B))$
correspond \`a celle de $\sigma$ dans $\Hom(\BG_m,\Out^0(A))$
via l'isomorphisme $\Out^0(A)\iso\Out^0(B)$ induit par $L$ 
(th\'eor\`eme \ref{Out0stable}).
Cela munit $B$ d'une structure gradu\'ee.

Pour $x\in k^\times$, on a 
$(A_{\sigma(x)})\otimes_A L\otimes_B B_{\tau(x)}\simeq L$.
Par cons\'equent, le $(A\otimes B^\circ)$-module $L$ est graduable
(proposition \ref{graduable}).
Ainsi, on a obtenu une
structure d'alg\`ebre gradu\'ee sur $B$ et une structure gradu\'ee sur
le $(A,B)$-bimodule $L$. On proc\`ede de m\^eme pour $L'$.
\end{proof}

\begin{rem}
Rappelons la construction en terme de bimodules.
Soit $X=A[t,t^{-1}]$, vu comme $A^\en[t,t^{-1}]$-module par
multiplication de $A[t,t^{-1}]$ \`a gauche et o\`u l'action de $a\in A_i$
\`a droite est la multiplication par $t^i a$.
Soit $Y=L'\otimes_A X\otimes_A L$. C'est un $B^\en[t,t^{-1}]$-module.
La preuve du lemme \ref{lemmestable} montre
l'existence de $B^\en[t,t^{-1}]$-modules
$Z$ et $P$ tels que $Z$ est localement libre de rang $1$
comme $B[t,t^{-1}]$-module et comme $B^\circ[t,t^{-1}]$-module et
$P$ est localement projectif comme $B^\en[t,t^{-1}]$-module, avec
$Y=P\oplus Z$. Alors, $Z$ est isomorphe au
$B^\en[t,t^{-1}]$-module associ\'e \`a $\tau$.
\end{rem}

On d\'emontre de la m\^eme mani\`ere, en utilisant le th\'eor\`eme
\ref{invarderiv}, le r\'esultat suivant:

\begin{thm}
Soient $A$ et $B$ deux $k$-alg\`ebres de dimension finie. On suppose $A$
munie d'une graduation.
Soient $L\in D^b(A\otimes B^\circ)$ et $L'\in D^b(B\otimes A^\circ)$
induisant des \'equivalences d\'eriv\'ees
(non gradu\'ees) inverses entre $A$ et $B$ (cf
\S \ref{secequivderiv}).

Alors, il existe une graduation sur $B$ et des structures de
complexes de bimodules gradu\'es sur $L$ et $L'$. 
En particulier, $L$ et $L'$ induisent des \'equivalences inverses entre
les cat\'egories $D^b(B\mmodgr)$ et $D^b(A\mmodgr)$.
\end{thm}

\subsubsection{Longueur}
Soient $A$ et $B$ deux $k$-alg\`ebres auto-injectives gradu\'ees
ind\'ecomposables et non simples. Soit $L\in (A\otimes B^\circ)\mmodgr$
induisant une \'equivalence stable.

\smallskip
Les longueurs des graduations
(cf \S \ref{dualite}) des deux alg\`ebres co\"{\i}ncident~:

\begin{lemme}
\label{invdual}
On a $n_A=n_B$.
\end{lemme}

\begin{proof}
Les bimodules gradu\'es $\Homgr_A(L,A)$ et
$\Homgr_B(L,B)$ sont isomorphes dans $(B\otimes A^\circ)\mstabgr$ (unicit\'e
\`a isomorphisme pr\`es d'un adjoint de $L\otimes_B-$).
Ces derniers sont respectivement isomorphes \`a
$L^*\langle -n_A\rangle$ et $L^*\langle -n_B\rangle$ (lemme \ref{dual}).
\end{proof}

\smallskip

\begin{lemme}
\label{invCartan}
On a $\det C_A\in \pm q^\BZ\cdot\det C_B$.
\end{lemme}

\begin{proof}
Cela r\'esulte de la commutativit\'e du diagramme de groupes de
Grothendieck (en fait de $\BZ[q,q^{-1}]$-modules)
$$\xymatrix{
K_0(A\mprojgr) \ar[r]^{C_A} \ar[d] & \ar[d] \ar[r] K_0(A\mmodgr) &
 \coker C_A \ar[d]^\sim\ar[r] & 0\\
K_0(B\mprojgr) \ar[r]_{C_B} & K_0(B\mmodgr) \ar[r] & \coker C_B \ar[r] & 0\\
}$$
o\`u les fl\`eches verticales sont induites par $L\otimes_A -$.
\end{proof}

\subsubsection{Nombre de modules simples}

Nous en arrivons maintenant au point crucial de cette \'etude num\'erique,
qui montre la pertinence de la prise en compte des graduations.

Si $A_0$ est de dimension globale finie, alors $B_0$ est aussi de
dimension globale finie (proposition \ref{dimproj}).

La proposition suivante fournit une r\'eponse positive \`a la conjecture
d'Alperin-Auslander dans le cadre d'alg\`ebres positivement gradu\'ees.

\begin{prop}
\label{invr}
Soient $A$ et $B$ deux $k$-alg\`ebres auto-injectives gradu\'ees
ind\'ecomposables. On suppose qu'il existe une \'equivalence stable
entre ces alg\`ebres induite par un bimodule gradu\'e.

Si $A_0$ est de dimension globale finie et si $A$ et $B$ sont
concentr\'ees en degr\'es positifs, alors $A$ et $B$ ont
le m\^eme nombre de modules simples.
\end{prop}

\begin{proof}
On sait que $n_A=n_B$ d'apr\`es le lemme \ref{invdual}.
Puisque $A_0$ et $B_0$ sont de dimension globale finie, leurs matrices
de Cartan sont inversibles sur $\BZ$.

D'apr\`es le lemme \ref{invCartan} et en utilisant la description du
d\'eterminant de la matrice de Cartan gradu\'ee donn\'ee par la proposition
\ref{Cartan}, on obtient $r_An_A=r_Bn_B$. Enfin, $n_B\not=0$, car $B\not=B_0$
puisque $B_0$ est de dimension globale finie (une alg\`ebre
sym\'etrique de dimension globale finie est semi-simple).
\end{proof}

\begin{rem}
Le probl\`eme de la positivit\'e d'une alg\`ebre stablement \'equivalente
\`a une alg\`ebre positivement gradu\'ee est d\'elicat. Il existe une
graduation sur l'alg\`ebre $\bar{\BF}_2 \GA_4$ et une
\'equivalence stable (et m\^eme d\'eriv\'ee) avec le bloc principal
$A$ de $\bar{\BF}_2\GA_5$ telles qu'il n'existe pas de graduation positive
sur $A$ compatible avec l'\'equivalence.
\end{rem}

Etant donn\'ee une alg\`ebre auto-injective $A$ gradu\'ee en degr\'es positifs,
dans quels cas peut-on reconstruire $A$ \`a partir de $A_0$? Il serait aussi
int\'eressant d'\'etudier le cas o\`u l'alg\`ebre diff\'erentielle
gradu\'ee $R\End^\bullet_A(A/A_{>0},A/A_{>0})$ est formelle.

\subsubsection{Rel\`evements d'\'equivalences stables}
\begin{defi}
Soit $A$ une $k$-alg\`ebre sym\'etrique ind\'ecomposable.
On dit que $A$ admet des
rel\`evements d'\'equivalences stables si toute \'equivalence stable de
type de Morita entre $A$ et une alg\`ebre sym\'etrique ind\'ecomposable
$B$ se rel\`eve en une \'equivalence d\'eriv\'ee. 
\end{defi}

\begin{rem}
Notons qu'on peut aussi demander la propri\'et\'e plus faible suivante: si
$B$ est une $k$-alg\`ebre sym\'etrique et s'il existe une \'equivalence
stable de type de Morita entre $A$ et $B$, alors il existe une \'equivalence
d\'eriv\'ee entre $A$ et $B$.

Notons enfin que la g\'en\'eralisation directe
au cas de corps non alg\'ebriquement clos n'est pas raisonnable: deux
extensions galoisiennes non-isomorphes du corps sont stablement \'equivalentes
mais non d\'eriv\'e-\'equivalentes.
\end{rem}

\subsection{Alg\`ebres de groupes et graduations}
\label{PE}

\subsubsection{Blocs locaux}
\label{se:blocs}
\paragraph{}
Soit $k$ un corps de caract\'eristique $p$.

Soit $P$ un $p$-groupe ab\'elien. L'alg\`ebre $kP$
est isomorphe \`a l'alg\`ebre gradu\'ee associ\'ee \`a la filtration
par le radical de $kP$~: $\bigoplus_i J^ikP/J^{i+1}kP$.

On obtient un isomorphisme en choisissant des g\'en\'erateurs
$\sigma_1,\ldots,\sigma_n$ de $P$ d'ordres $d_1,\ldots,d_n$ tels que
$P=\prod_i \langle \sigma_i\rangle$.
Alors, 
$$kP\iso k[x_1,\ldots,x_r]/(x_i^{d_i})\iso \bigoplus_i J^ikP/J^{i+1}kP$$
$$\sigma_i-1\mapsto x_i\mapsto \overline{\sigma_i-1}$$
o\`u $\overline{\sigma_i-1}$ est l'image de $\sigma_i-1\in JkP$
dans $\bigoplus_i J^ikP/J^{i+1}kP$.

\smallskip

Soit $E$ un $p'$-sous-groupe d'automorphismes de $P$, $\hat{E}$
une extension centrale de $E$ par $k^\times$ et $A=k_*P\rtimes \hat{E}$
l'alg\`ebre de groupe de $P\rtimes E$ tordue par l'extension centrale.
Alors, $A$ est isomorphe \`a $\bigoplus_i J^iA/J^{i+1}A$.

Un tel isomorphisme fournit une graduation sur $A$. On a $A_0=k_*\hat{E}$.

\medskip
\paragraph{}
Reprenons la construction plus explicitement.

Soit $P$ un $p$-groupe ab\'elien homocyclique d'exposant $p^r$.

Choisissons un sous-espace $V$ de $J(kP)$ tel que 
$J(kP)=V\oplus J(kP)^2$. Alors, $kP$ est engendr\'ee, comme
$k$-alg\`ebre, par $1$ et $V$. Le noyau du morphisme canonique
$S(V)\to kP$ est $V^{p^r}$.
On d\'efinit donc une graduation de $kP$ en prenant $k\cdot 1$
en degr\'e $0$ et $V$ en degr\'e $1$.

\medskip
\paragraph{}
Soit $E$ un $p'$-groupe agissant sur $P$.
On a une suite exacte scind\'ee de $kE$-modules
$$0\to J(kP)^2\to J(kP)\to J(kP)/J(kP)^2\to 0.$$
Soit $V$ un sous-$kE$-module de $J(kP)$ tel que
$J(kP)=V\oplus J(kP)^2$.
Alors, on munit $kP$ de la graduation
o\`u $V$ est en degr\'e $1$ et $k\cdot 1$ en degr\'e $0$~:
c'est une graduation $E$-invariante.

Donnons-nous en outre une $k^\times$-extension centrale $\hat{E}$ de $E$.
Consid\'erons maintenant $A=k_*(P\rtimes \hat{E})$. Alors, on munit $A$ de la
graduation telle que $k_*\hat{E}$ est en degr\'e $0$ et $V$ en degr\'e $1$.

\medskip
\paragraph{}
Prenons maintenant $P$ un $p$-groupe ab\'elien quelconque et $E$
un $p'$-groupe agissant sur $P$. Puisque les $\BZ_pE$-modules
ind\'ecomposables de torsion sont homocycliques, il existe une
d\'ecomposition $E$-stable $P=\prod_i P_i$, o\`u $P_i$ est
homocyclique d'exposant $p^i$.

Pour chaque $i$, on fixe $V_i$ un sous-$kE$-module de $J(kP_i)$ tel que
$J(kP_i)=V_i\oplus J(kP_i)^2$. Soit $V=\oplus_i V_i$.
On munit alors $kP$ de la graduation o\`u $V$ est en degr\'e $1$.
Comme plus haut, \'etant donn\'ee une $k^\times$-extension centrale
$\hat{E}$ de $E$, on munit $k_*(P\rtimes \hat{E})$ de la
graduation telle que $k_*\hat{E}$ est en degr\'e $0$ et $V$ en degr\'e $1$.

\subsubsection{Application}
Soit $G$ un groupe fini, $k$ un corps alg\'ebriquement clos de
caract\'eristique $p$, $A$ un bloc de $kG$ de d\'efaut $D$. Soit
$B$ le bloc correspondant de $N_G(D)$. D'apr\`es \cite{Ku}, l'alg\`ebre
$B$
est Morita-\'equivalente \`a une alg\`ebre de groupe tordue
$k_* D\rtimes\hat{E}$, o\`u $E=N_G(D,b_D)/C_G(D)$, $\hat{E}$
est une extension centrale de $E$ par $k^\times$ et $(D,b_D)$
est une $A$-sous-paire maximale.

\begin{thm}
\label{th:grblocs}
Si $D$ est cyclique, ab\'elien \'el\'ementaire de rang $2$ ou
si $D$ est ab\'elien d'ordre $8$,
alors il existe une graduation sur $A$ et une \'equivalence stable
gradu\'ee entre $A$ et $B$.
\end{thm}

\begin{proof}
D'apr\`es \cite[Theorem 6.3 et Theorem 6.10]{Roucha} et \cite{Rou5}
(blocs non principaux), il existe une \'equivalence stable de type de Morita
entre $A$ et $B$. Le th\'eor\`eme r\'esulte alors du th\'eor\`eme 
\ref{th:gradstable}.
\end{proof}

Le cas des blocs \`a d\'efaut cyclique a \'et\'e \'etudi\'e en d\'etail par
D.~Bogdanic \cite{Bog}.

\medskip
La conjecture de Brou\'e \cite{Bro}
sur les blocs \`a d\'efaut ab\'elien pr\'edit que
les conclusions du th\'eor\`eme \ref{th:grblocs}
sont vraies si $D$ est ab\'elien 
(la conjecture pr\'edit plus pr\'ecis\'ement que les blocs seront
d\'eriv\'e-\'equivalents). Si cette conjecture est vraie, alors tous
les blocs \`a d\'efaut ab\'elien ont une graduation non triviale.

\begin{rem}
Puisque la conjecture de Brou\'e est prouv\'ee pour les groupes
sym\'etriques \cite{ChRou}, on obtient des graduations
sur les blocs \`a d\'efaut ab\'elien des groupes sym\'etriques. En 
transf\'erant la graduation canonique sur les ``bons blocs'', on 
obtient des graduations dont il est naturel de conjecturer que les
matrices de Cartan gradu\'ees associ\'ees sont d\'ecrites par des
polyn\^omes de Kazhdan-Lusztig paraboliques. Le transfert des graduations
provient d'op\'erateurs du type ``tresses \'el\'ementaires''. En fait,
\cite{Rou4} montre que ces graduations sont
compatibles aux foncteurs induction et restriction appropri\'es. La
somme des cat\'egories de repr\'esentations de groupes sym\'etriques sur $k$
fournit une $2$-repr\'esentation ``simple'' de $\hat{\Gsl}_p$ qui se
retrouve automatiquement munie de graduations. Celles-ci se d\'ecrivent
de mani\`ere explicite via un isomorphisme avec des alg\`ebres cyclotomiques
de Hecke de carquois \cite{BruKl,Rou4}.
\end{rem}

\begin{question}
On suppose $D$ ab\'elien.
Existe-t'il une graduation en degr\'es positifs sur $A$,
compatible avec une \'equivalence d\'eriv\'ee avec $B$, muni de sa
graduation canonique? Peut-on trouver une telle graduation telle que
$R\End^\bullet_A(A/A_{>0},A/A_{>0})$ est une alg\`ebre diff\'erentielle
gradu\'ee formelle?
\end{question}

L'existence de graduations positives est connue si $D$ est cyclique
\cite{Bog}. Elle est vraie aussi si $D\simeq (\BZ/2)^2$.

\begin{question}
\label{qu:relevabel}
Soit $P$ un $p$-groupe ab\'elien, $E$ un sous-groupe d'ordre premier \`a
$p$ du groupe d'automorphismes de $P$ et consid\'erons une extension centrale
$\hat{E}$ de $E$ par $k^\times$. Est-ce que
$k_*P\rtimes E$ satisfait les rel\`evements d'\'equivalences
stables?
\end{question}

Une r\'eponse affirmative \`a une forme plus pr\'ecise de la question
\ref{qu:relevabel}
a pour cons\'equence une r\'eponse affirmative \`a la conjecture de Brou\'e
\cite{Roucha,Rou5}.

\subsection{Extensions triviales d'alg\`ebres}

\subsubsection{}
Soit $B$ une $k$-alg\`ebre de dimension finie.
On d\'efinit l'alg\`ebre $T(B)=B\oplus B^*$ avec le produit
$$(a,f)\cdot (b,g)=(ab,ag+fb)$$
o\`u on a utilis\'e la structure de $(B,B)$-bimodule de $B^*$. 

L'alg\`ebre $T(B)$ est gradu\'ee, avec $B$ en degr\'e $0$ et $B^*$ en
degr\'e $1$.
On a une forme lin\'eaire canonique sur $T(B)$~:
$$t:T(B)\to k\langle -1\rangle,\ \ \ (a,f)\mapsto f(1)$$
qui induit une structure d'alg\`ebre sym\'etrique sur $T(B)$.

\smallskip
Une alg\`ebre gradu\'ee extension triviale se reconna\^{\i}t ais\'ement~:

\begin{prop}
\label{pr:rectrivial}
Soit $A$ une $k$-alg\`ebre sym\'etrique gradu\'ee en degr\'es $0$ et $1$
munie d'une forme sym\'etrisante $t:A\to k\langle -1\rangle$.
Alors,
$$\phi=(\id,\hat{t}_{|A_1}):A=A_0\oplus A_1\to T(A_0)=A_0\oplus A_0^*$$
est un isomorphisme d'alg\`ebres gradu\'ees.
\end{prop}

\begin{proof}
Soit $a\in A_0$ et $b\in A_1$. Alors, 
$\phi(ba)=(1,\hat{t}(ba))=(1,\hat{t}(b)a)=(1,\hat{t}(b))(a,1)=\phi(b)\phi(a)$ et
$\phi(ab)=(1,\hat{t}(ab))=(1,a\hat{t}(b))=\phi(a)\phi(b)$.
\end{proof}

\subsubsection{}
Le th\'eor\`eme suivant \'etend la proposition \ref{pr:A0ff}. Il
est d\^u \`a Happel \cite[Theorem 10.10]{Ha} lorsque
$A=T(A_0)$ et $A_0$ est de dimension globale finie.

\begin{thm}
\label{th:Happel}
Soit $A$ une $k$-alg\`ebre gradu\'ee auto-injective avec $A_{<0}=0$. Si
$\soc A\subset A_{>0}$, alors le foncteur
compos\'e
$$D^b(A_0\mMod)\xrightarrow{A/A_{>0}\otimes_{A_0}^\BL-}
D^b(A\mModgr)\xrightarrow{\can} A\mstabgr$$
est pleinement fid\`ele.

Si en outre $A=A_0\oplus A_1$ et $A_0$ est de dimension globale
finie, alors le foncteur $D^b(A_0\mMod)\to A\mstabgr$ est une \'equivalence.
\end{thm}

\begin{proof}
Notons que $A$ n'a pas de
module simple projectif, puisque $\soc A\subset A_{>0}$.
Soient $M$ et $N$ deux $A_0$-modules simples. Soit $d>0$. Alors,
$\soc\Omega^d N$ est en degr\'es $>0$. Par cons\'equent,
$\Hom_{A\mModgr}(M,\Omega^d N)=0$.

On a $\Hom_{D^b(A\mModgr)}(M,N[d])\iso \Hom_{A\mstabgr}(M,\Omega^{-d}N)$.
Cet isomorphisme reste vrai pour $d=0$ car $A$ n'a pas de
module simple projectif. On en d\'eduit que
$\Hom_{D^b(A_0\mMod)}(M,N[d])\iso \Hom_{A\mstabgr}(M,\Omega^{-d}N)$ pour
tout $d$. Puisque les modules simples engendrent $D^b(A_0\mMod)$
comme cat\'egorie triangul\'ee, la premi\`ere partie du th\'eor\`eme est 
\'etablie.

\smallskip
Supposons maintenant $A=A_0\oplus A_1$ et $A_0$ est de dimension globale
finie. Alors, $A_1$ est un cog\'en\'erateur injectif.
Soit $\CT$ la sous-cat\'egorie \'epaisse de $D^b(A\mModgr)$ engendr\'ee
par $A_0$ et par les $A\langle i\rangle$ pour $i\in\BZ$.
On a $A_1\langle -1\rangle\simeq \Omega(A_0)\in\CT$, donc $M\langle -1\rangle\in\CT$
pour tout $M\in A_0\mMod$. Par r\'ecurrence, on en d\'eduit que
$M\langle -i\rangle\in\CT$ pour tout $i\ge 0$ et tout $M\in A_0\mMod$.

On a $A_1\in \CT$, donc $A_0\langle 1\rangle\simeq \Omega^{-1}(A_1)\in\CT$.
Par cons\'equent, 
$M\langle 1\rangle\in\CT$ pour tout $M\in A_0\mMod$ et
par r\'ecurrence, on d\'eduit que
$M\langle i\rangle\in\CT$ pour tout $i\ge 0$ et tout $M\in A_0\mMod$.

Finalement, $\CT=D^b(A\mModgr)$, donc $A\mstabgr$ est engendr\'ee par
$A_0$. Ceci montre que le foncteur
$D^b(A_0\mMod)\to A\mstabgr$ est une \'equivalence.
\end{proof}

Rappelons le r\'esultat suivant de Rickard \cite[Theorem 3.1]{Ri}:
\begin{thm}
\label{th:Rickard}
Soient $A_0$ et $B_0$ deux $k$-alg\`ebres de dimension finie et
$F:D^b(A_0\mMod)\iso D^b(B_0\mMod)$ une \'equivalence de cat\'egories
triangul\'ees. Alors il existe une \'equivalence de cat\'egories
triangul\'ees gradu\'ees $G:D^b(T(A_0)\mModgr)\iso D^b(T(B_0)\mModgr)$ rendant
le diagramme suivant commutatif
$$\xymatrix{
D^b(A_0\mMod)\ar[r]^-F_-{\sim} \ar[d]_\can &
D^b(B_0\mMod) \ar[d]^\can \\
D^b(T(A_0)\mModgr)\ar[r]_-G^{\sim} &  D^b(T(B_0)\mModgr)
}$$
\end{thm}

\begin{cor}
\label{co:relevpos}
Soient $A_0$ une $k$-alg\`ebre de dimension finie et de dimension globale finie.
Soit $B$ une $k$-alg\`ebre sym\'etrique gradu\'ee ind\'ecomposable
telle que $B_{<0}=0$.

Toute \'equivalence triangul\'ee gradu\'ee
$T(A_0)\mstabgr\iso B\mstabgr$ se rel\`eve en une
\'equivalence triangul\'ee gradu\'ee
$D^b(T(A_0)\mModgr)\iso  D^b(B\mModgr)$.
\end{cor}

\begin{proof}
Le lemme \ref{invdual} montre que $B_i=0$ pour $i>1$ et que
$B$ admet une forme sym\'etrisante $t_B:N\to k\langle -1\rangle$.
La proposition \ref{pr:rectrivial} fournit alors un isomorphisme
d'alg\`ebres gradu\'ees $B\iso T(B_0)$.
Les th\'eor\`emes \ref{th:Happel} et \ref{th:Rickard} fournissent
la conclusion.
\end{proof}

Rappelons qu'une cat\'egorie ab\'elienne $\CC$ est h\'er\'editaire si
$\Ext^i_\CC=0$ pour $i\ge 2$. Si $\CC$ est une cat\'egorie $k$-lin\'eaire
dont les $\Hom$ sont de dimension finie, on appelle foncteur de Serre
un foncteur $S:\CC\to\CC$ muni d'isomorphismes bifonctoriels
$$\Hom(M,N)^*\iso \Hom(N,S(M))$$
pour $M,N\in\CC$.

Le r\'esultat suivant est proche de r\'esultats d\^us \`a Asashiba
\cite{As1,As2} (Asashiba suppose que $T(A_0)$ est de type de
repr\'esentation fini).

\begin{thm}
\label{th:relevhered}
Soit $A_0$ une $k$-alg\`ebre de dimension finie ind\'ecomposable.
Supposons qu'il existe
une cat\'egorie ab\'elienne h\'er\'editaire $\CC$ telle que
$D^b(A_0\mMod)\simeq D^b(\CC)$. Alors, $T(A_0)$ admet des rel\`evements
d'\'equivalences stables.
\end{thm}

\begin{proof}
Soit $B$ une $k$-alg\`ebre sym\'etrique ind\'ecomposable
munie d'une \'equivalence stable de type de Morita avec $T(A_0)$. D'apr\`es
le th\'eor\`eme \ref{Out0stable}, il existe une graduation sur $B$ compatible
avec l'\'equivalence. En composant avec l'\'equivalence du th\'eor\`eme
\ref{th:Happel} et celle fournie par l'hypoth\`ese, on obtient une
\'equivalence
$$F:B\mstabgr\iso D^b(\CC).$$

 Soient $S_n=S_0,S_1,\ldots,S_{n-1}$ des $B$-modules
simples de degr\'e $0$ et soient $d_0,\ldots,d_{n-1}\in\BZ$ tels que
$\Ext^1(S_i,S_{i+1}\langle -d_i\rangle)\not=0$. Puisque $F(S_i)$
est un objet ind\'ecomposable de $D^b(\CC)$ et que $\CC$ est
h\'er\'editaire, il existe un entier $n_i$ tel que 
$F(S_i)\simeq M_i[n_i]$, o\`u $M_i=H^{-n_i}(F(S_i))\in\CC$.

Le foncteur de Serre de $B\mstabgr$ est $\Omega\langle 1\rangle$. On a
donc
\begin{align*}
\Ext^1(S_i,S_{i+1}\langle -d_i\rangle)&\simeq
\Hom_{B\mstabgr}(S_i,\Omega^{d_i-1}S^{-d_i}(S_{i+1}))\\
&\simeq
\Hom_{D^b(\CC)}(M_i,S^{-d_i}(M_{i+1})[n_{i+1}-n_i+1-d_i]).
\end{align*}

Le lemme \ref{le:Serrehered} ci-dessous montre que l'homologie de
$S^{-d_i}(M_{i+1})$ s'annule en degr\'es $>d_i$. Puisque l'espace des
$\Hom$ ci-dessus est non nul, on en d\'eduit que $n_i-n_{i+1}+2d_i\ge 0$.
Par cons\'equent,
$\sum_{i=0}^{n-1} d_i\ge 0$. La proposition
\ref{caracpositivite} montre alors que la graduation sur $B$ peut \^etre choisie
en degr\'es $\ge 0$. Le corollaire \ref{co:relevpos} fournit la conclusion.
\end{proof}

\begin{lemme}
\label{le:Serrehered}
Soit $\CC$ une cat\'egorie ab\'elienne h\'er\'editaire et $S$ un
foncteur de Serre pour $D^b(\CC)$. Soit $M\in D^b(\CC)$, $r\in\BZ$ et
$d\in \BZ_{\ge 0}$.
Si $H^i(M)=0$ pour $i>r$, alors $H^i(S^{-d}(M))=0$ pour $i>r+d$.
\end{lemme}

\begin{proof}
Soit $n$ maximum tel que $N=H^n(S^{-1}M)\not=0$. 
On a $\Hom(S^{-1}M,N[-n])\simeq \Hom(N,M[n])^*\not=0$. Par cons\'equent,
$r-n\ge -1$. On en d\'eduit le lemme par r\'ecurrence sur $d$.
\end{proof}

\subsection{Alg\`ebres ext\'erieures}
\subsubsection{}
Soit $k$ un corps alg\'ebriquement clos et $V$ un espace vectoriel
de dimension finie sur
$k$. Soit $G$ un groupe fini d'ordre inversible dans $k$ et
$\rho:G\to \GL(V)$ une repr\'esentation.

Soit $A=\Lambda(V)\rtimes G$ le produit crois\'e:
$A=\Lambda(V)\otimes kG$ comme espace vectoriel, $\Lambda(V)\otimes k$
et $k\otimes kG$ sont des sous-alg\`ebre et $gvg^{-1}=g(v)$ pour
$g\in G$ et $v\in V$.
L'alg\`ebre $A$ est gradu\'ee: $kG$ est en degr\'e $0$ et $V$ en degr\'e $1$.

\smallskip
Fixons un isomorphisme d'espaces vectoriels
$\Lambda^nV\iso k$, o\`u $n=\dim V$.
Soit $\nu\in\GL(V)$ la multiplication par $(-1)^{n+1}$. On note
encore $\nu$ l'automorphisme d'alg\`ebre induit de $\Lambda(V)$.
On \'etend enfin $\nu$ \`a un automorphisme d'alg\`ebre de $A$ par
$\nu(g)=\det(g)g$.

\smallskip
On d\'efinit la forme lin\'eaire $t:A\to k\langle -n\rangle$ par
$t(x\otimes g)=\delta_{1g}t(x)$ pour $g\in G$ et $x\in \Lambda^n V$ et
par $t(a)=0$ si $a\in A_{<n}$. On a $t(ab)=t(b\nu(a))$ pour
$a,b\in A$.

On a un accouplement parfait
$$A\times A\to k\langle -n\rangle,\ \ (a,b)\mapsto t(ab)$$
et
$$\hat{t}:A_\nu\langle n\rangle\to A^*,\ \ b\mapsto (a\mapsto t(ab))$$
est un isomorphisme de $A^\en$-modules gradu\'es.
En particulier, $A$ est une alg\`ebre de Frobenius.

\begin{prop}
L'alg\`ebre $A$ est sym\'etrique si et seulement si
$\rho(G)\le\SL(V)$ et $(-1)^{\dim V+1}\in \rho(G)$.
\end{prop}

\begin{proof}
L'alg\`ebre $A$ est sym\'etrique si et seulement si $\nu$ est int\'erieur.

Supposons $\rho(G)\le\SL(V)$ et soit $g\in G$ tel que
$\rho(g)=(-1)^{n+1}\in G$. Alors, $\nu=\ad(g)$.

Supposons $\nu$ int\'erieur. L'action de $\nu$ sur $A_0$ est int\'erieure, donc
$\rho(G)\le\SL(V)$. Supposons $n$ pair.
Soit $a\in A^\times$ tel que $\ad(a)=\nu$.
L'action de $\ad(a)$ sur $V\iso A_{\le 1}/A_0$
ne d\'epend que de l'image $\bar{a}$ de $a$ dans $(kG)^\times\iso
A^\times/(1+JA)$. On a $\rho(\bar{a})=(-1)^{n+1}$. D\'ecomposons
$\bar{a}=\sum_{g\in G}\alpha_g g$ avec $\alpha_g\in k$. Soit $v\in V$. On a
$\ad(a)(v)=-v$, donc $av=-va$, \ie,
$\sum_g \alpha_g g(v)\otimes g=-\sum_g \alpha_g v\otimes g$ et finalement
$g(v)=-v$ si $\alpha_g{\not=}0$. Soit $g\in G$ tel que
$\alpha_g{\not=0}$. Alors, $\rho(g)=-1$.
\end{proof}

\begin{question}
\label{qu:relevext}
Supposons $\rho(G)\le\SL(V)$ et $(-1)^{\dim V+1}\in\rho(G)$. Est-ce que
$\Lambda(V)\rtimes G$ satisfait les rel\`evements d'\'equivalences
stables?
\end{question}

La question pr\'ec\'edente est importante pour la th\'eorie des
repr\'esentations modulaires des groupes finis. Supposons $k$ 
de caract\'eristique $2$. 
Soit $P$ un $2$-groupe ab\'elien \'el\'ementaire et $G$ un groupe
d'automorphismes d'ordre impair de $P$.
La construction du \S \ref{se:blocs} fournit un isomorphisme de $k$-alg\`ebres
$kP\rtimes G\simeq \Lambda(V)\rtimes G$, o\`u
$V=P\otimes_{\BF_2}k$. Par cons\'equent, une r\'eponse affirmative \`a la
question \ref{qu:relevext} pour $(V,G)$ implique une r\'eponse affirmative
\`a la question \ref{qu:relevabel} pour $(P,G)$.

\subsubsection{Dimension $2$}
Supposons maintenant $\rho$ injective, $\rho(G)\le\SL(V)$,
$G$ n'est pas un groupe cyclique d'ordre impair et
la caract\'eristique de $k$
n'est pas $2$. Il existe alors $\sigma\in G$ tel que
$\rho(\sigma)=-1$.

Soit $I$ l'ensemble des classes d'isomorphisme de repr\'esentations
irr\'eductibles de $G$ sur $k$. 
Soit $I_0$ le sous-ensemble
de $I$ des repr\'esentations triviales sur $\sigma$.
Soit $\Delta$ le carquois de sommets $I$ avec $\dim\Hom(\chi,\psi\otimes
\rho)$ fl\`eches de $\psi$ vers $\chi$, lorsque $\psi\in I-I_0$ et
$\chi\in I$.

\begin{lemme}
\label{le:extastrivial}
On a un isomorphisme d'alg\`ebres
$\Lambda(V)\rtimes G\simeq T(C)$, o\`u $C$ est Morita-\'equivalente
\`a l'alg\`ebre de carquois de $\Delta$.
\end{lemme}

\begin{proof}
Soit $e_\chi$ l'idempotent primitif
de $Z(kG)$ associ\'e \`a $\chi\in I$.  On d\'efinit
$A'=\End_A(\bigoplus_{\chi\in I_0} Ae_\chi\oplus
\bigoplus_{\chi\in I-I_0} Ae_\chi\langle 1\rangle)$. On a
$A'=A$ comme alg\`ebres non gradu\'ees, mais $A'$ a une graduation
diff\'erente. On a une \'equivalence de cat\'egories
ab\'eliennes gradu\'ees $A'\mModgr\iso A\mModgr$ (\S \ref{se:chgtgrad}).
On a $A'=A'_0\oplus A'_2$. On consid\`ere pour finir la
graduation sur $B=A$ obtenue en divisant les degr\'es par $2$:
$B_0=A'_0$ et $B_1=A'_2$ et $B$ a une forme sym\'etrisante
$B\to k\langle -1\rangle$. L'alg\`ebre $B_0$ est Morita-\'equivalente
\`a l'alg\`ebre du carquois $\Delta$.
D'apr\`es la proposition \ref{pr:rectrivial}, on a $B\simeq T(B_0)$.
\end{proof}

\begin{thm}
\label{th:relevext}
Si $\dim V=2$, $k$ n'est pas de caract\'eristique $2$,
$\rho$ est injective
et $G\le\SL(V)$ n'est pas un groupe cyclique d'ordre impair, alors
$\Lambda(V)\rtimes G$ satisfait les rel\`evements d'\'equivalences
stables.
\end{thm}

\begin{proof}
Le th\'eor\`eme r\'esulte du lemme \ref{le:extastrivial} ci-dessous et
du th\'eor\`eme \ref{th:relevhered}.
\end{proof}

\begin{rem}
Une alg\`ebre de Brauer associ\'ee \`a une ligne avec multiplicit\'e
$1$ est isomorphe \`a une extension triviale de l'alg\`ebre d'un carquois
de type $A$ o\`u aucun sommet n'est la source (resp. le but) de deux fl\`eches
distinctes \cite[Proposition 8.1]{Bog}. Le th\'eor\`eme
\ref{th:relevhered} montre alors qu'une telle alg\`ebre de Brauer (et donc
toute alg\`ebre de Brauer \`a multiplicit\'e triviale) satisfait
la propri\'et\'e de rel\`evement. C'est un r\'esultat classique et l'approche
d'Asashiba \cite{As1} de ce r\'esultat est proche de celle que nous
avons adopt\'ee.
\end{rem}

\end{document}